\documentclass[12pt]{article}
\usepackage{latexsym,lscape}
\usepackage{algorithm}
\usepackage{algpseudocode}
\usepackage{amsmath}
\usepackage{amsfonts}
\usepackage{dsfont}
\usepackage{amssymb}
\usepackage{booktabs}
\usepackage{mathrsfs}
\usepackage{multirow}
\usepackage{amsthm}
\usepackage{subcaption}
\usepackage{color}
\usepackage{mathdots}
\usepackage{cases}
\usepackage{graphicx}	
\usepackage{rotating}
\usepackage{verbatim}
\usepackage{enumitem}
\usepackage{url}
\usepackage{natbib}
\usepackage{appendix}
\setcitestyle{aysep={}}

\definecolor{pred}{RGB}{148,55,61}
\definecolor{blue}{rgb}{0,0,0.9}
\definecolor{red}{rgb}{0.9,0,0}
\definecolor{green}{rgb}{0,0.9,0}

\newtheorem{theorem}{Theorem}[section]
\newtheorem{proposition}{Proposition}[section]
\newtheorem{assumption}{Assumption}[section]
\newtheorem{lemma}{Lemma}[section]

\def\Re{\mathbb{R}}
\def\S{\mathbb{S}}
\newcommand{\blind}{0}
\begin{document}

\def\spacingset#1{\renewcommand{\baselinestretch}%
{#1}\small\normalsize} \spacingset{1}


\if0\blind
{
\title{\bf An Efficient Linearly Convergent Regularized Proximal Point Algorithm for Fused Multiple Graphical Lasso Problems}
\date{February 19, 2019}
\author{
{Ning Zhang}\thanks{Department of Applied Mathematics, The Hong Kong Polytechnic University, Hung Hom, Hong Kong ({\tt ningzhang\_2008@yeah.net}).}
\quad {Yangjing Zhang}\thanks{({Corresponding author}) Department of Mathematics, National University of Singapore, 10 Lower Kent Ridge Road, Singapore 119076 ({\tt zhangyangjing@u.nus.edu}).}
\quad {Defeng Sun}\thanks{Department of Applied Mathematics, The Hong Kong Polytechnic University, Hung Hom, Hong Kong ({\tt defeng.sun@polyu.edu.hk}). {This author is supported by Hong Kong Research Grant Council grant PolyU153014/18p}.}
\quad{Kim-Chuan Toh}\thanks{Department of Mathematics, and Institute of Operations Research and Analytics, National University of Singapore, 10 Lower Kent Ridge Road, Singapore 119076 ({\tt mattohkc@nus.edu.sg}).
The research of this author is supported in part by the Ministry of Education, Singapore, Academic Research Fund (Grant number:
R-146-000-257-112).
}}
\maketitle
 \fi

\if1\blind
{
  \bigskip
  \bigskip
  \bigskip
  \begin{center}
    {\LARGE\bf An Efficient Linearly Convergent Regularized Proximal Point Algorithm for Fused Multiple Graphical Lasso Problems}
\end{center}
  \medskip
} \fi

\bigskip
\begin{abstract}
Nowadays, analysing data from different classes or over a temporal grid has attracted a great deal of interest.
As a result, various multiple graphical models for learning a collection of graphical models simultaneously have been derived by introducing sparsity in graphs and similarity across multiple graphs. This paper focuses on
the fused multiple graphical Lasso model which encourages not only shared pattern of sparsity, but also shared values of edges across different graphs.
For solving this model, we develop an efficient regularized proximal point algorithm, where the subproblem in each iteration of the algorithm is solved by a superlinearly convergent semismooth Newton method. To implement the semismooth Newton method, we derive an explicit expression for the generalized Jacobian of the proximal mapping of the fused multiple graphical Lasso regularizer. Unlike those widely used first order methods, our approach has heavily exploited the underlying second order information through the semismooth Newton method. This can not only accelerate the convergence of the algorithm, but also improve its robustness. The efficiency and robustness of our proposed algorithm are demonstrated {by comparing} with some state-of-the-art methods on both synthetic and real data sets.  {Supplementary materials for this article are available online.}
\end{abstract}

\noindent%
\vfill

\newpage
\spacingset{1.5} 
\section{Introduction}
Undirected graphical models have been especially popular for learning  {conditional independence structures} among a large number of variables where the observations are drawn independently and identically from the same distribution. The Gaussian graphical model is one of the most widely used undirected graphical  {models}. In the high-dimensional and low-sample-size settings, it is always assumed that the conditional independence structure or the precision matrix is sparse in a certain sense. In other words, its corresponding undirected graph is expected to be sparse. To promote sparsity, there has been a great deal of interest in using the $\ell_1$ norm penalty in statistical applications \citep{banerjee2008model,friedman2008sparse,rothman2008sparse}. In many conventional applications, a single Gaussian graphical model is typically enough to capture the conditional independence structure of the random variables. However, due to the heterogeneity or similarity of the data involved, it is increasingly appealing to fit a collection of such models jointly, such as inferring the time-varying networks and finding the change-points  \citep{ahmed2009recovering,monti2014estimating,gibberd2017regularized,hallac2017network,yang2018estimating} and estimating multiple precision matrices simultaneously for variables from distinct but related classes \citep{guo2011joint,danaher2014joint,yang2015fused}.

Multiple graphical models refer to the models that can estimate a collection of precision matrices jointly. Specifically, let $\Delta^{(l)}$ be $L$ random vectors (from different classes or over a temporal grid) drawn independently from different distributions $\mathcal{N}_p(\mu^{(l)},\Sigma^{(l)}),\,l=1,2,\dots,L,\,L\geq 2$. Assume that the multivariate random variable $\Delta^{(l)}$ has $N_l$ observations $\delta^{(l)}_1,\delta^{(l)}_2,\dots,\delta^{(l)}_{N_l}$, for each $l\in\{1,2,\dots,L\}$. Then the sample means are $ \bar{\mu}^{(l)} = \frac{1}{N_l}\sum_{i=1}^{N_l}\delta^{(l)}_i $ and the sample covariance matrices are $S^{(l)}=\frac{1}{N_l -1}\sum_{i=1}^{N_l}(\delta^{(l)}_i - \bar{\mu}^{(l)} ) (\delta^{(l)}_i - \bar{\mu}^{(l)} )^T$, $l=1,2,\dots,L$. The multiple graphical model for estimating {the} precision matrices $(\Sigma^{(l)})^{-1},\,l=1,2,\dots,L$ jointly is the model with  {the variable} $\Theta=(\Theta^{(1)},\dots,\Theta^{(L)})\in\S^p\times\cdots\times\S^p$:
\begin{equation}\label{model-MGL}
\begin{array}{cl}
\min\limits_{\Theta} & \displaystyle\sum^L_{l=1} \left(-\log \det \,\Theta^{(l)}+\langle S^{(l)},\Theta^{(l)} \rangle \right)+ \mathcal{P}(\Theta),
\end{array}
\end{equation}
where $\mathcal{P}$ is a penalty function, which usually promotes sparsity in each $\Theta^{(l)}$ and similarities among different $\Theta^{(l)}$'s.
Various penalties have been considered in the literature \citep{ahmed2009recovering,guo2011joint,danaher2014joint,monti2014estimating,yang2015fused,gibberd2017regularized}.

In this paper, we focus on the following fused graphical Lasso (FGL) regularizer which was used by \cite{ahmed2009recovering} and \cite{yang2015fused}:
\begin{equation}\label{regularizer-FGL}
\begin{array}{l}
 \mathcal{P}(\Theta) \displaystyle = \lambda_1 \sum^L_{l=1} \sum_{i\neq j}|\Theta^{(l)}_{ij}| + \lambda_2 \sum^{L}_{l=2} \sum_{i\neq j} |{\Theta^{(l)}_{ij}} - {\Theta^{(l-1)}_{ij}}|.
 \end{array}
\end{equation}
We refer to problem \eqref{model-MGL} with the FGL regularizer $\mathcal{P}$  in \eqref{regularizer-FGL} as the FGL problem.
The FGL regularizer  is in some sense a generalized fused Lasso regularizer \citep{tibshirani2005sparsity}. It applies the $\ell_1$ penalty to all the off-diagonal elements of the $L$ precision matrices and the consecutive differences of the elements of successive precision matrices. Many elements with the same indices in the estimated matrices $\Theta^{(1)},\dots,\Theta^{(L)}$ will be close or even identical when the parameter $\lambda_2$ is large enough. Therefore, the FGL regularizer encourages not only shared pattern of sparsity, but also shared values across different graphs.

Existing algorithms for solving the FGL problem are quite limited in the literature. One of the most extensively used algorithms for solving this class of problems is the alternating direction method of multipliers (ADMM) \citep{danaher2014joint,hallac2017network,gibberd2017regularized}.
Besides, a proximal Newton-type method \citep{hsieh2011sparse,lee2014proximal} was implemented by \cite{yang2015fused} for solving the FGL problem.
As we know, ADMM could be a  {practical} first order method for finding approximate solutions of low or moderate accuracy. However, ADMM hardly utilizes any second order information, which generally must be used in order to obtain highly accurate solutions. Although the proximal Newton-type method does incorporate some forms of second order information, a complicated quadratic approximation problem has to be solved in each iteration, and this computation is usually time-consuming. It is worth mentioning that the regularizers are often introduced to promote certain structures
in the estimated precision matrices, and the trade-off between biases and variances in the resulting {estimators} is controlled by the regularization parameters \citep{fan2010selective}. But in practice, it is extremely hard to find the optimal regularization parameters. Therefore, a sequence of regularization parameters is  {applied in practice}, and consequently, a sequence of corresponding optimization problems  {must} be solved \citep{fan2013tuning}. Under such a circumstance, a highly efficient and robust algorithm for solving the FGL model becomes particularly important.

In this paper, we will design a semismooth Newton (SSN) based regularized proximal point algorithm (rPPA) for solving the FGL problem, which is inspired by \cite{li2017efficiently}, where they have  convincingly demonstrated the superior numerical performance of the SSN based augmented Lagrangian method (ALM), known as {\sc Ssnal}, for solving the fused Lasso problem \citep{tibshirani2005sparsity}.
Thanks to the fact that the FGL problem has close connections to the fused Lasso problem, many of the virtues and theoretical insights of the {\sc Ssnal}\, for solving the fused Lasso problem can be observed in our approach. However, we should emphasize that solving the FGL problem is much more challenging than solving the fused Lasso problem.
Specifically, the difficulties are mainly due to the log-determinant function $\log\det\,(\cdot)$ and the matrix  variables,
as described below.
\vskip 1mm
\begin{itemize}[topsep=1pt,itemsep=-.6ex,partopsep=1ex,parsep=1ex,leftmargin=4ex]
\item[(a)] Unlike the simple quadratic functions in the fused
 Lasso problem, the function $\log\det\,(\cdot)$ is defined on the space of positive definite matrices. Therefore, the FGL model requires the positive definiteness of their solutions. This greatly increases  the difficulty and complexity of theoretical analysis and numerical implementation.

\item[(b)] \cite{li2017efficiently} constructed  an efficiently computable element in the generalized Jacobian of the proximal mapping of the fused Lasso regularizer, which is an essential step for solving the fused Lasso problem. Based on the constructions, we could obtain an efficiently computable generalized Jacobian of the proximal mapping of the FGL regularizer. However, this process needs more complicated manipulations of coordinates for a collection of matrix variables, unlike the vector case of the fused Lasso problem.
\end{itemize}
The key issue in the implementation of rPPA for solving the FGL model is the
computation of the solution of the subproblem in each rPPA iteration.
For this purpose, we will design an SSN method to solve those subproblems.
We note that the numerical performance of the SSN method relies
critically on the efficient calculation of the generalized Jacobian of the proximal mapping of the FGL regularizer and that of the log-determinant function. Fortunately, the generalized Jacobian of the proximal mapping of the FGL regularizer can be constructed efficiently based on that of the proximal mapping of the fused Lasso regularizer given by \cite{li2017efficiently}.
As a result, the generalized Jacobian of the proximal mapping of the FGL regularizer would inherit the structured sparsity (referred to as second order sparsity) from that of the fused Lasso regularizer. Due to the structured sparsity, the computation of a matrix-vector product in the SSN method is reasonably cheap and thus the SSN method is quite efficient for solving each subproblem.
To summarize, it can be proven that our rPPA for solving the FGL problem has a linear convergent guarantee, and the convergence rate can be arbitrarily fast by choosing a sufficiently large proximal penalty parameter. Moreover, the SSN method for solving each of rPPA subproblems can be shown to be superlinearly convergent. Thus, based on these excellent convergent properties and the novel exploitation of the second order sparsity, we can expect the SSN based rPPA for solving the FGL problem to be highly efficient.  Indeed, our numerical experiments have confirmed the high
efficiency and robustness of the proposed algorithm for solving the FGL problems accurately.

The remaining parts of this paper are as follows. Section 2 presents some definitions and preliminary results.
In section 3, we present a semismooth Newton based regularized proximal point algorithm for solving the FGL problem and its convergence properties. The numerical performance of our proposed algorithm on time-varying stock prices data sets and categorical text data sets are evaluated in section 4. Section 5 gives the conclusion.

\medskip
\noindent
{\bf Notations.}
$\S^p_+$ ($\S^p_{++}$) denotes the cone of positive semidefinite (definite) matrices in the space of $p\times p$ real symmetric matrices $\S^p$. For any $A,\,B\in \S^p$, we denote $A \succeq B$ $(A\succ B)$ if $A-B \in \S^p_+$ ($A-B \in \S^p_{++}$). In particular, $A \succeq 0$ $(A\succ 0)$ indicates $A \in \S^p_+$ ($A \in \S^p_{++})$.
We let $\mathcal{X} := \S^p_+\times\cdots\times\S^p_+$ and
 $\mathcal{Y} =: \S^p\times\cdots\times\S^p$  {to} be the Cartesian product of $L$ positive semidefinite cones $ \S^p_+ $ and that of $L$ spaces of symmetric matrices $\S^p$, respectively.
$\Re^n$ denotes the $n$-dimensional Euclidean space, and
 $\Re^{m\times n}$ denotes the set of all $m\times n$ real matrices.
For any $x\in\Re^n$, $\|x\|_1 = \sum_{i=1}^{n}|x_i|$, and $\|x\| = \sqrt{\sum_{i=1}^{n}|x_i|^2}$.
We use the {\sc Matlab} notation $[A; B ]$ to denote the matrix obtained by appending $B$ below the last row of $A$, when the number of columns of $A$ and $B$ is identical.
For any matrix $A\in\Re^{m \times n}$, $A_{ij}$ denotes the $(i,j)$-th element of $A$. For any $X := (X^{(1)},\dots,X^{(L)})\in\mathcal{Y}$, $X_{[ij]} :=[X^{(1)}_{ij};\ldots;X^{(L)}_{ij}]\in\Re^L$ denotes the column vector obtained by taking out the $(i,j)$-th elements across all $L$ matrices $X^{(l)},\,l=1,\dots,L$.
${\rm Diag}(D_1,\ldots,D_n)$ denotes the block diagonal matrix whose $i$-th diagonal block are the matrix $D_i$, $i=1,\dots,n$.
$I_n$ denotes the $n\times n$ identity matrix, and $I$ denotes an identity matrix or map when the dimension is clear from the context.
The function composition is denoted by $\circ$, that is, for any functions $f$ and $g$, $(f\circ g)(\cdot) := f(g(\cdot))$. The Hadamard product is denoted by $\odot$.

\section{Preliminaries}

Let $\mathcal{E}$ be a finite-dimensional real Hilbert space, and $\Xi:\, \mathcal{E}\rightarrow{\Re}\cup\{+\infty\}$ be a proper and closed convex function. The Moreau-Yosida regularization \citep{moreau1965proximite,yosida1980functional} of $\Xi$ is defined by
\begin{equation}\label{def-MY}
\begin{array}{l}
\Psi_{\Xi}(u):=\min_{u'}\left\{\Xi(u')+\frac{1}{2}\|u'-u\|^2\right\},\,\,\forall u\in \mathcal{E}.
\end{array}
\end{equation}
The proximal mapping associated with $\Xi$ is the unique minimizer of \eqref{def-MY} defined by
\begin{equation}\label{def-prox}
\begin{array}{l}
{\rm Prox}_{\Xi}(u) := \arg\min_{u'}\left\{\Xi(u')+\frac{1}{2}\|u'-u\|^2\right\},\,\,\forall\,u\in\mathcal{E}.
\end{array}
\end{equation}
Moreover, $\Psi_{\Xi}(\cdot)$ is a continuously differentiable convex function  \citep{lemarechal1997practical,rockafellar2009variational}, and its gradient is given by
\begin{equation}\label{def-grad-MY}
\begin{array}{l}
\nabla \Psi_{\Xi} (u) =u-{\rm Prox}_{\Xi}(u) ,\,\,\forall\,u\in\mathcal{E}.
\end{array}
\end{equation}
For notational convenience, define $\vartheta:\mathbb{S}^p\rightarrow \Re\cup\{+\infty\}$ by
$$
\vartheta(A) =\left\{\begin{array}{ll}
-{\log\det}(A), & \hbox{if $A\in\mathbb{S}^p_{++}$};\\
+\infty, & \hbox{otherwise}.
\end{array}
\right.
$$
 {Let  $\beta > 0$ be given.} Define two scalar functions as follows:
$$
\begin{array}{l}
\phi^+_{\beta}(x):=(\sqrt{x^2+4\beta} + x)/{2},\,\,
\phi^-_{\beta}(x):=(\sqrt{x^2+4\beta} - x)/{2},\,\,\forall\, x\in\Re.
\end{array}
$$
In addition, the matrix counterparts of these two scalar functions can be defined by
$$
\phi^+_{\beta}(A):= Q{\rm Diag}(\phi^+_{\beta}(d_1),\dots,\phi^+_{\beta}(d_p))Q^T,\,\,\,
\phi^-_{\beta}(A):= Q{\rm Diag}(\phi^-_{\beta}(d_1),\dots,\phi^-_{\beta}(d_p))Q^T
$$
for any $A\in \S^p$ with its eigenvalue decomposition $A=Q{\rm Diag}(d_1,d_2,\ldots,d_p)Q^T $, where $d_1\geq d_2\geq\cdots\geq d_p$. It is easy to show that $\phi^+_{\beta}$ and $\phi^-_{\beta}$ are well-defined. Moreover, $\phi^+_{\beta}(A)$ and $\phi^-_{\beta}(A)$ are positive definite for any $A\in\S^p$.
\begin{proposition}\label{phiprime}\cite[Lemma 2.1 (b)]{wang2010solving}
   The function $\phi^+_{\beta}:\,\S^p\to \S^p$ is continuously differentiable, and its directional derivative $(\phi^+_{\beta})'(A)[B]$ at $A$ for any  $B\in\S^p$ is given by
$$ (\phi^+_{\beta})'(A)[B] = Q[\Gamma \odot (Q^T B Q)]Q^T,$$
where $A$ admits the eigenvalue decomposition $A=Q{\rm Diag}(d_1,d_2,\dots,d_p)Q^T,\,d_1\geq d_2\geq\cdots\geq d_p$, and $\Gamma\in\S^p$ is defined by $$\Gamma_{ij} = (\phi^+_{\beta}(d_i)+\phi^+_{\beta}(d_j))/(\sqrt{d_i^2+4\beta}+\sqrt{d_j^2+4\beta})
,\,\,i,j=1,2,\dots,p.$$
\end{proposition}

\begin{proposition}\cite[Proposition 2.3]{yang2013proximal}\label{prop:logdet}
  For any $A\in\S^p$, it holds that
$
\begin{array}{l}
{\rm Prox}_{\beta \vartheta}(A) = \phi^+_{\beta}(A)\,\,\hbox{and}\,\,
\Psi_{\beta \vartheta}(A)=-\beta\log\det(\phi^+_{\beta}(A))+\frac{1}{2}\|\phi^-_{\beta}(A)\|^2.
\end{array}
$
\end{proposition}

\subsection{Surrogate Generalized Jacobian of ${\rm Prox}_{\mathcal{P}}$}
In this section, we analyse
the proximal mapping of the regularizer $\mathcal{P}$ defined by \eqref{regularizer-FGL}. For any $\Theta\in\mathcal{Y}$, one might observe that the penalty term $\mathcal{P}(\Theta)$ merely penalizes the off-diagonal elements, and it is the same fused Lasso regularizer that acts on each vector $\Theta_{[ij]}\in\Re^L,\,i\neq j$. It holds that $ \mathcal{P}(\Theta)=\sum_{i\neq j} \varphi(\Theta_{[ij]})\,\,\hbox{with}\,\, \varphi(x)=\lambda_1\|x\|_1+\lambda_2\|Bx\|_1,\,\forall\, x\in\Re^L. $
The function $\varphi$ is the fused Lasso regularizer, and
the matrix $B\in \Re^{(L-1) \times L}$ is defined by $Bx = [x_1 - x_2; \dots; x_{L-1} - x_L]$, $\forall\, x\in\Re^L$.
The formula for the generalized Jacobian of $ {\rm Prox}_{\varphi} $ has been derived by \cite{li2017efficiently} and will be used in our subsequent algorithmic design.
Define the surrogate generalized Jacobian $\widehat{\partial}{\rm Prox}_{\mathcal{P}}(X):\mathcal{Y}\rightrightarrows\mathcal{Y}$ of ${\rm Prox}_{\mathcal{P}}$ at $X$ as follows:
\begin{align}\label{jac-prox-p-FGL}
\left\{\begin{array}{l}
\mathcal{W}\in\widehat{\partial}{\rm Prox}_{\mathcal{P}}(X)\hbox{ if and only if there exist $M^{(ij)}\in\widehat{\partial}{\rm Prox}_{\varphi}(X_{[ij]})$, $ i < j$}\\
\hbox{such that}\ (\mathcal{W}[Y])_{[ij]}=\left\{\begin{array}{ll}
M^{(ij)}Y_{[ij]}, & \hbox{if $ i < j$},\\
Y_{[ii]}, & \hbox{if $i = j$},\\
M^{(ji)}Y_{[ij]}, & \hbox{if $ j < i$},
\end{array}
\right.\,\,i,j=1,\ldots,p,\,\,\,\forall\, Y\in\mathcal{Y},
\end{array}
\right.
\end{align}
where $\widehat{\partial}{\rm Prox}_{\varphi}(\cdot)$ is the surrogate generalized Jacobian of ${\rm Prox}_{\varphi}$ \cite[Equation 22]{li2017efficiently} and will be given in the supplementary materials. From Theorem 1 by \cite{li2017efficiently}, one can  obtain the following theorem, which justifies why $\widehat{\partial}{\rm Prox}_{\mathcal{P}}(X)$ in \eqref{jac-prox-p-FGL} can be used as the surrogate generalized Jacobian of ${\rm Prox}_{\mathcal{P}}$ at $X$.
\begin{theorem}\label{thm-semismooth-FGL}
The surrogate generalized Jacobian $\widehat{\partial}{\rm Prox}_{\mathcal{P}}(\cdot)$ defined in \eqref{jac-prox-p-FGL} is a nonempty compact valued, upper semicontinuous multifunction.
Given any $ X\in\mathcal{Y}$,
any element in the set $\widehat{\partial}{\rm Prox}_{\mathcal{P}}(X)$ is self-adjoint and positive semidefinite.
Moreover, there exists a neighborhood $\mathcal{U}_X $ of $X$ such that for all $Y \in \mathcal{U}_X$,
\begin{equation*}
{\rm Prox}_{\mathcal{P}}(Y) - {\rm Prox}_{\mathcal{P}}(X) - \mathcal{W}[Y-X] =0,\,\,\forall\,\mathcal{W}\in\widehat{\partial}{\rm Prox}_{\mathcal{P}}(Y).
\end{equation*}
\end{theorem}

\subsection{Lipschitz Continuity of the Solution Mapping }
By introducing an auxiliary variable $ \Omega =(\Omega^{(1)},\ldots,\Omega^{(L)} ) \in\mathcal{Y} $, we can rewrite problem \eqref{model-MGL} equivalently as
\begin{equation}\label{model-MGL-2}
\min\limits_{\Theta,\,\Omega} \Big\{f(\Theta,\Omega):=\sum^L_{l=1} \left(\vartheta(\Omega^{(l)})+\langle S^{(l)},\Theta^{(l)} \rangle \right)+ \mathcal{P}(\Theta)\,\big|\, \Theta -\Omega =0 \Big\}.
\end{equation}
The Lagrangian function of the above problem is given by
\begin{equation*}
\mathcal{L}(\Theta,\Omega,X)=f(\Theta,\Omega) - \langle \Theta-\Omega,\,X\rangle,\,\,\forall\,(\Theta,\Omega,X)\in \mathcal{X}\times\mathcal{X}\times\mathcal{Y}.
\end{equation*}
The dual problem of \eqref{model-MGL-2} takes the following form \cite[Theorem 3.3.5]{borwein2010convex}:
\begin{equation}\label{model-MGL-dual}
\begin{array}{cl}
\max\limits_{X} & \displaystyle\sum^L_{l=1} \left(-\vartheta(X^{(l)})+ p\right)- \mathcal{P}^*(X-S).
\end{array}
\end{equation}
The Karush-Kuhn-Tucker (KKT) optimality conditions \citep{han2018linear}
for \eqref{model-MGL-2} are given as follows:
\begin{equation}\label{def-KKT-MGL}
\Theta-{\rm Prox}_{\mathcal{P}}(\Theta+ X-S)=0,\,\,
\Omega^{(l)}-{\rm Prox}_{\vartheta}(\Omega^{(l)}-X^{(l)})=0,\,\,l=1,\ldots,L,\,\,\Theta - \Omega=0.
\end{equation}
We make the following assumption on the existence of solutions to the KKT system.
\begin{assumption}\label{assum}
The solution set to the KKT system \eqref{def-KKT-MGL} is nonempty.
\end{assumption}
Define an operator $\mathcal{T}_{\mathcal{L}}$ by
$
\mathcal{T}_{\mathcal{L}}(\Theta,\Omega,X):=\{(\Theta',\Omega',X')\,|\,(\Theta',\Omega',-X')\in \partial\mathcal{L}(\Theta,\Omega,X) \}.
$
Since the function $\vartheta(\cdot)$ is strictly convex, under Assumption \ref{assum}, we can see that the KKT system \eqref{def-KKT-MGL} has a unique KKT point, denoted by $(\overline{\Theta},\overline{\Omega},\overline{X})$, and $\mathcal{T}^{-1}_{\mathcal{L}}(0)=\{(\overline{\Theta},\overline{\Omega},\overline{X})\}$.
\begin{proposition}\label{prof:metric-subreg}
There exists a nonnegative scalar $\kappa$ such that for some $\varrho>0$ it holds that
$
\|(\Theta,\Omega,X)-(\overline{\Theta},\overline{\Omega},\overline{X})\|\leq \kappa \|\Delta\|,\,\,\forall\, \Delta\in\mathcal{T}_{\mathcal{L}}((\Theta,\Omega,X))\,\,\hbox{and}\,\, \|\Delta\|\leq \varrho.
$
\end{proposition}
\begin{proof}
Note that the regularizer $\mathcal{P}$ defined by \eqref{regularizer-FGL} is a positive homogeneous function. Therefore, it follows from Example 11.4(a) by \cite{rockafellar2009variational} that the conjugate function $\mathcal{P}^*$ is an indicator function of a nonempty convex polyhedral set. This, together with Theorem 2.7 by \cite{li2018highly} and Proposition 6 by \cite{cui2018r},  {proves  the required result}.
\end{proof}

\section{Regularized Proximal Point Algorithm}
In this section, we present a regularized proximal point algorithm (rPPA) for solving the problem \eqref{model-MGL-2} with the FGL regularizer defined by \eqref{regularizer-FGL}.
Given a sequence of positive scalars $\sigma_{k}\uparrow\sigma_{\infty}\leq \infty$, the $k$-th iteration of
PPA for solving \eqref{model-MGL-2} is given by
\begin{equation}\label{alg-PPA0}
\begin{array}{l}
(\Theta^{k+1},\Omega^{k+1})\approx\arg\underset{\Theta,\,\Omega}{\min} \Big\{f(\Theta,\Omega)+ \frac{1}{2\sigma_k}(\|\Theta- \Theta^k\|^2
+ \|\Omega - \Omega^k\|^2)\,|\,\Theta-\Omega=0\Big\},
\end{array}
\end{equation}
where $k\geq 0$ and $f$ is the objective function of the problem \eqref{model-MGL-2}.

There are many ways to solve \eqref{alg-PPA0}. Inspired by recent progresses in solving  large scale convex optimization problems \citep{yang2013proximal,li2018highly,li2017efficiently,Zhang2018efficient}, we shall adopt the approach of solving \eqref{alg-PPA0} via employing a sparse SSN method to its dual.
The dual of \eqref{alg-PPA0} takes the following form:
$$
\begin{array}{l}
\sup\limits_{X}\Big\{ \Phi_k(X):=\inf\limits_{\Theta,\Omega} \big\{\mathcal{L}(\Theta,\Omega,X) + \frac{1}{2\sigma_k}(\|\Theta- \Theta^k\|^2
+ \| \Omega - \Omega^k\|^2)\big\}\Big\}.
\end{array}
$$
By the definition of the Moreau-Yosida regularization \eqref{def-MY}, we can write $\Phi_k(\cdot)$ explicitly as follows:
$$
\begin{array}{l}
\quad \Phi_k(X)\\
= \inf\limits_{\Theta}
\big\{ \mathcal{P}(\Theta) + \langle \Theta, S - X \rangle
+ \frac{1}{2\sigma_k} \|\Theta - {\Theta}^k \|^2 \big\} \\[2mm]
~~+{\sum^L_{l=1} }\inf_{\Omega^{(l)}}
\big\{\vartheta(\Omega^{(l)}) + \langle \Omega^{(l)},X^{(l)} \rangle
+ \frac{1}{2\sigma_k} \| \Omega^{(l)}- ({\Omega}^{(l)})^k \|^2 \big\}\\[2mm]
=\frac{1}{\sigma_k}\Psi_{\sigma_k\mathcal{P}}({\Theta}^k + \sigma_k(X-S))+\sum^L_{l=1} \frac{1}{\sigma_k} \Psi_{\sigma_k\vartheta}(({\Omega}^k)^{(l)} - \sigma_k X^{(l)}) \\[2mm]
~~- \frac{1}{2\sigma_k}\|{\Theta}^k + \sigma_k(X-S)\|^2 + \frac{1}{2\sigma_k}\|{\Theta}^k\|^2- {\sum^L_{l=1}}\big(\frac{1}{2\sigma_k} \| ({\Omega}^k)^{(l)} -\sigma_k X^{(l)} \|^2 - \frac{1}{2\sigma_k}\| ({\Omega}^{(l)})^k\|^2\big).
\end{array}
$$
Therefore, by Proposition \ref{prop:logdet} and the definition of the proximal mapping \eqref{def-prox}, the $k$-th iteration of PPA \eqref{alg-PPA0} can be written as
$$
\left\{
\begin{array}{l}
 \Theta^{k+1} = {\rm Prox}_{\sigma_k \mathcal{P}}(\Theta^k + \sigma_k (X^{k+1}-S)),\\
 (\Omega^{(l)})^{k+1} = {\rm Prox}_{\sigma_k\vartheta}\big((\Omega^{(l)})^{k} - \sigma_k(X^{(l)})^{k+1}\big),\,\,l = 1,2,\dots,L,
 \end{array}
\right.
$$
where $X^{k+1}$ approximately solves the following problem:
$
\begin{array}{l}
X^{k+1} \approx \arg\max\limits_{X}~\Phi_k(X).
\end{array}
$
Since $\Phi_k(\cdot)$ is not strongly concave in general,
we consider the following rPPA.
\begin{algorithm}[H]
\caption{A regularized proximal point algorithm (rPPA) for solving \eqref{model-MGL-2}}\label{alg-ppa}
Choose $\Theta^0\in\mathcal{X},\,\Omega^0\in \mathcal{X}$. Iterate the following steps for $k=0,1,2,\dots$:
\begin{itemize}[topsep=1pt,itemsep=-.6ex,partopsep=1ex,parsep=0.5ex,leftmargin=8.5ex]
\item[{\bf Step 1.}] Compute
\begin{equation}\label{PPA-sub}
\begin{array}{c}
X^{k+1} \approx \arg\max\limits_{X}~\Big\{\widehat{\Phi}_k(X):=\Phi_k(X)-\frac{1}{2\sigma_k}\|X-X^k\|^2\Big\}.
\end{array}
\end{equation}
 \item[{\bf Step 2.}] Compute $\Theta^{k+1} = {\rm Prox}_{\sigma_k \mathcal{P}}(\Theta^k + \sigma_k (X^{k+1}-S))$ and for $l = 1,\dots,L$,
 $$
 (\Omega^{(l)})^{k+1} = {\rm Prox}_{\sigma_k\vartheta}\big((\Omega^{(l)})^{k} - \sigma_k(X^{(l)})^{k+1}\big)= \phi_{\sigma_k}^+\big((\Omega^{(l)})^{k} - \sigma_k(X^{(l)})^{k+1}\big).
 $$
 \item[{\bf Step 3.}]Update $\sigma_{k+1}\uparrow\sigma_{\infty} \leq \infty$.
\end{itemize}
\end{algorithm}	
Since in practice the inner subproblem \eqref{PPA-sub} can only be solved inexactly,
we will use the following standard stopping criteria studied by \cite{rockafellar1976monotone}:
$$
\begin{array}{rl}
{\rm (A)} & \|\nabla\widehat{\Phi}_k(X^{k+1})\|
\;\leq\; \varepsilon_k/\sigma_k,\,\varepsilon_k\geq 0,\,\sum_{k=0}^{\infty}\varepsilon_k<\infty;
\\[7pt]
{\rm (B)} & \|\nabla \widehat{\Phi}_k(X^{k+1})\|
\;\leq\; (\delta_k /\sigma_k)\|(\Theta^{k+1},\Omega^{k+1}) - (\Theta^{k},\Omega^k)\|,\,\, \delta_k\geq 0,\sum_{k=0}^{\infty}\delta_k<\infty.
\end{array}
$$
The reason for using the above stopping criteria is due to the fact that Algorithm \ref{alg-ppa} is equivalent to the primal-dual PPA in the sense of \cite{rockafellar1976monotone}. Moreover, we have the following convergence results.
\begin{theorem}
Let $\{ (\Theta^{k},\Omega^{k},X^{k}) \}$ be an infinite sequence generated by Algorithm \ref{alg-ppa} under stopping criterion {\rm (A)}. Then the sequence $\{(\Theta^k,\Omega^k)\}$ converges to the unique solution $(\overline{\Theta},\overline{\Omega})$ of \eqref{model-MGL-2}, and the sequence $ \{X^k\} $ converges to the unique solution $\overline{X}$ of \eqref{model-MGL-dual}.
Furthermore, if the criterion ${\rm (B)}$ is also executed in Algorithm \ref{alg-ppa}, there exists $\bar{k}\geq 0$ such that for all $k\geq \bar{k}$,
\begin{equation*}
\|(\Theta^{k+1},\Omega^{k+1},X^{k+1})-(\overline{\Theta},\overline{\Omega},\overline{X})\|\leq \mu_k\|(\Theta^{k},\Omega^{k},X^{k})-(\overline{\Theta},\overline{\Omega},\overline{X})\|,
\end{equation*}
where the convergence rate
$$
1>\mu_k:=[\kappa(\kappa^2+\sigma^2_k)^{-1/2}+\delta_k]/(1-\delta_k)\rightarrow \mu_{\infty}=\kappa(\kappa^2+\sigma^2_{\infty})^{-1/2}\,\,(\mu_{\infty}=0\,\hbox{if}\,\,\sigma_{\infty}=\infty)
$$
and the parameter $\kappa$ is from Proposition \ref{prof:metric-subreg}.
\end{theorem}
\begin{proof}
The global convergence of Algorithm \ref{alg-ppa} can be obtained from  {Theorem 1} by \cite{rockafellar1976monotone} and the uniqueness of the KKT point.
The linear rate of convergence can be derived from Proposition \ref{prof:metric-subreg} and Theorem 2 by \cite{rockafellar1976monotone}.
\end{proof}

\subsection{Semismooth Newton Method for Solving Subproblem \eqref{PPA-sub}}
From \eqref{def-grad-MY}, Proposition \ref{prop:logdet}, and Theorem 31.5 by \cite{rockafellar2015convex}, we know that $\widehat{\Phi}_k$ is a continuously differentiable, strongly concave function and
$$
\nabla{\Phi_k}(X) = - {\rm Prox}_{\sigma_k \mathcal{P}}\big(U_k(X)\big)+\big(
\phi_{\sigma_k}^+(W^{(1)}_k(X)), \ldots, \phi_{\sigma_k}^+(W^{(L)}_{k}(X)) \big),
$$
where $U_k(X):={\Theta}^k + \sigma_k(X-S)$ and $W^{(l)}_k(X):=(\Omega^{k})^{(l)} - \sigma_kX^{(l)}$,\,\,$l=1,2,\ldots,L$.
Therefore, one can obtain the unique solution to  problem \eqref{PPA-sub} by solving the nonsmooth system
\begin{equation} \label{newton-equation}
\nabla\widehat{\Phi}_k(X)=\nabla\Phi_k(X)-(X-X^k)/\sigma_k= 0.
\end{equation}
Recall that $\phi_{\sigma_k}^+(\cdot)$ is differentiable and its derivative is given by Proposition \ref{phiprime}. Thus, the surrogate generalized Jacobian $\widehat{\partial}(\nabla\Phi_k)(X)$ of $ \nabla\Phi_k$ at $X$ is defined as follows:
\begin{equation}
\left\{\begin{array}{l}
\mathcal{V}\in \widehat{\partial}(\nabla\Phi_k)(X) \hbox{ if and only if there exists $\mathcal{G}\in\widehat{\partial}{\rm Prox}_{\mathcal{P}}(U_{k}(X)/\sigma_k)$ such that}\\
\mathcal{V}[D] = - \sigma_k \mathcal{G}[D]\\
~~~~~~~~~~- \sigma_k \big((\phi^+_{\sigma_k})'(W_{_k}^{(1)}(X))[D^{(1)}],
\dots,(\phi^+_{\sigma_k})'(W_{k}^{(L)}(X))[D^{(L)}]\big),\,\,\forall D\in\mathcal{Y}.
\end{array}
\right.
\end{equation}
With the generalized Jacobian of $\nabla\Phi_k$, we are ready to solve  equation \eqref{newton-equation} by the  SSN method, where the Newton systems are solved inexactly by the conjugate gradient (CG) method.
\begin{algorithm}[H]
\caption{A semismooth Newton (SSN) method for solving \eqref{newton-equation}}
\label{alg-ssn}
Given $\mu \in (0,1/2)$, $\bar{\eta} \in (0,1)$, $\tau \in (0,1]$,
and $\rho \in (0,1)$. Choose $X^0\in \S^p_{++}\times\cdots\times\S^p_{++}$. Iterate the following steps for $ j = 0,1, \dots$:
\begin{itemize}[topsep=1pt,itemsep=-.6ex,partopsep=1ex,parsep=0.5ex,leftmargin=8.5ex]
\item[{\bf Step 1.}] (Newton direction) Choose one specific map $\mathcal{V}_j\in \widehat{\partial}(\nabla\Phi_k)(X^j)$.
 Apply the CG method to find an approximate solution $D^j$ to
$$
 (\mathcal{V}_j-{\sigma^{-1}_k}{I})[D] = -\nabla \widehat{\Phi}_k(X^j)
$$
 such that $\| (\mathcal{V}_j-{\sigma^{-1}_k}{I})[D^j]+ \nabla \widehat{\Phi}_k(X^j)\| \leq \min(\bar{\eta},\|\nabla \widehat{\Phi}_k(X^j)\|^{1+\tau}).$
\item[{\bf Step 2.}] (Line search) Set $ \alpha_j = \rho^{m_j}$, where $m_j$ is the smallest nonnegative integer $m$ for which
$$
\widehat{\Phi}_k(X^j + \rho^m D^j) \geq \widehat{\Phi}_k(X^j) + \mu\rho^m \langle \nabla \widehat{\Phi}_k(X^j),D^j \rangle.
$$
\item[{\bf Step 3.}] Set $ X^{j+1} = X^j + \alpha_j D^j $.	
\end{itemize}
\end{algorithm}	
\noindent
Next, we derive the convergence result of the  SSN method (Algorithm \ref{alg-ssn}).
\begin{theorem}
Let $\{X^j\}$ be the infinite sequence generated by Algorithm \ref{alg-ssn}. Then $\{X^j\}$ converges to the unique optimal solution $\widehat{X}$ of \eqref{newton-equation} and $\|X^{j+1}-\widehat{X}\|=\mathcal{O}(\|X^{j}-\widehat{X}\|^{1+\tau})$.
\end{theorem}
\begin{proof}
Since the proximal mapping ${\rm Prox}_{\mathcal{P}}$ is piecewise linear and Lipschitz continuous, we know from Theorem 7.5.17 by \cite{facchinei2007finite} that ${\rm Prox}_{\mathcal{P}}$ is directionally differentiable. This, together with Theorem \ref{thm-semismooth-FGL}, implies that ${\rm Prox}_{\mathcal{P}}$ is strongly semismooth with respect to the multifunction $\widehat{\partial}{\rm Prox}_{\mathcal{P}}$ (for its definition, see e.g., Definition 1 by \cite {li2017efficiently}). Therefore, the conclusion follows from the strong convexity of $\widehat{\Phi}_k(\cdot)$, Proposition \ref{phiprime}, and Proposition 7 \& Theorem 3 by \cite{li2017efficiently}.
\end{proof}

\section{Numerical Experiments}
In this section, we compare the performance of our algorithm rPPA with the alternating direction method of multipliers (ADMM) and the proximal Newton-type method implemented  by \cite{yang2015fused}
(referred to as MGL here) for which the solver is available at \url{http://senyang.info/}.

The following paragraph describes the measurement of the accuracy of an approximate optimal solution and the stopping criteria of the three methods.
Since both rPPA and ADMM can generate primal and dual approximate solutions, we can assess the accuracy of their solutions by the relative KKT residuals. Unlike the primal-dual method, MGL merely gives the primal solution and the KKT residual of a solution generated by MGL is not available. Instead, we measure the relative error of the objective value obtained by MGL with respect to that computed by rPPA. Based on the KKT optimality condition \eqref{def-KKT-MGL}, the accuracy of an approximate optimal solution $(\Theta,\Omega,X)$ generated by rPPA (Algorithm \ref{alg-ppa}) is measured by defining the following relative residuals:
$$
\begin{array}{l}
\eta_P:=\max \left\{
\frac{\|\Theta-{\rm Prox}_{\mathcal{P}}(\Theta + X - S)\|}{1+\|\Theta\|},\,
\frac{\|\Theta - \Omega \|}{1+\|\Theta\|},\,
\underset{1 \leq l \leq L}{\max}\Big\{ \frac{\|\Omega^{(l)} X^{(l)}-I\|}{1+\sqrt{p}} \Big\}
\right\}.
\end{array}
$$
Likewise, the accuracy of an approximate optimal solution $(\Theta,X,Z)$ generated by ADMM is measured by the relative KKT residual $\eta_A$ (defined in the supplementary material) that is analogous to $\eta_P$.

In our numerical experiments, we terminate rPPA if it satisfies the condition
$\eta_P< \varepsilon$ for a given accuracy tolerance $\varepsilon$; similarly for
ADMM with the stopping condition $\eta_A < \varepsilon.$
Note that the terminating condition for MGL is different.
Let ``${\rm pobj}_P$'' and ``${\rm pobj}_M$'' be the primal objective function values computed by rPPA and MGL, respectively.
MGL will be terminated when the relative difference of its objective value with respect to  {that} obtained by rPPA is smaller than the given tolerance $\varepsilon$, i.e.,
\begin{equation}\label{delta-M}
\Delta_M:= ({\rm pobj}_{M} - {\rm pobj}_{P})/(1+|{\rm pobj}_{M}|+|{\rm pobj}_{P}|) < \varepsilon.
\end{equation}

We adopt a warm-start strategy to initialize rPPA. That is, we first run ADMM (with identity matrices as the starting point) for a fixed number of iterations to generate a good initial point to warm-start rPPA. We also stop ADMM as soon as the relative KKT residual of the computed iterate is less than $100 \varepsilon$. Note that such a warm-starting strategy is sound since in the
initial phase of rPPA where the iterates are not close to the optimal solution (as measured by the associated relative KKT residual), it is computationally
wasteful to use the more expensive rPPA iteration when the fast local linear convergence behavior of the algorithm has yet to kick in. Under such
a scenario, naturally one would use cheaper iterations such as those of ADMM to generate the approximate
solution points until the relative KKT residual has been sufficiently reduced.

For the tuning parameters $\lambda_1$ and $\lambda_2$,  for each test instance we select three pairs
that lead to reasonable sparsity. In the following tables, ``P'' stands for rPPA; ``A'' stands for ADMM; ``M'' stands for MGL; ``nnz'' denotes the number of nonzero entries in the solution $\Theta$ obtained by rPPA using the  estimation:
$ {\rm nnz}:= \min \{k\,|\,\sum_{i=1}^{k}|\hat{x}_i| \geq 0.999\|\hat{x}\|_1\}, $
where $\hat{x}\in\Re^{p^2L}$ is the vector obtained via sorting all elements in $\Theta$ by magnitude in a descending order; ``density'' denotes the quantity nnz$/(p^2L)$.
The time is displayed in the format of ``hours:minutes:seconds'', and the fastest method in terms of running time is highlighted in red.
The errors  presented in the tables are the relative KKT residuals $\eta_P$ for rPPA and $\eta_A$ for ADMM; while the error for MGL is $\Delta_M$ in \eqref{delta-M}.
\subsection{Nearest-neighbour Networks}
In this section, we assess the effectiveness of the FGL model on a simulated network: nearest-neighbour network.
The nearest-neighbour network is generated by modifying the data generation mechanism described by \cite{li2006gradient}.
We set $p=500$ and $L=3$. For each $l = 1,2,\dots,L$, we generate 10,000 independently and identically distributed observations from a multivariate Gaussian distribution $\mathcal{N}_p(0,(\Omega^{(l)})^{-1})$, where $\Omega^{(l)}$ is the precision matrix of the $l$-th class. The details of the generation of $\Omega^{(l)}$ are as follows. First of all, $p$ points are randomly generated on a unit square, their pairwise distances are calculated, and $m$-nearest neighbours of each point in terms of distance are found. The nearest-neighbour network is obtained by linking any two points that are $m$-nearest neighbours of each other. The integer $m$ controls the degree of sparsity of the network, and we set $m=5$ in our simulation. Subsequently, we add heterogeneity to the common structure by further creating individual links as follows: for each $\Omega^{(l)}$, a pair of symmetric zero elements is randomly selected and replaced with a value uniformly drawn from the interval $[-1,-0.5]\cup[0.5,1]$. This procedure is repeated ceil$(M/4)$ (the nearest integer greater than or equal to $M/4$) times, where $M$ is the number of edges in the nearest-neighbour graph. {In our simulation, the true number of edges in the three networks is $3690$.

There is a pair of tuning parameters $\lambda_1$ and $\lambda_2$ which must be specified. In the FGL model, $\lambda_1$ drives sparsity and $\lambda_2$ drives similarity, and we say that $\lambda_1$ and $\lambda_2$ are the sparsity and similarity control parameters respectively.
In order to show the diversity of sparsity in our experiments, we choose a series of $\lambda_1$ for the FGL model with $\lambda_2$ fixed.
Figure~\ref{fig-nn} shows the relative ability of the FGL model to recover the network structures and to detect the change-points.

Figure~\ref{figsub-1} displays the number of true positive edges selected (i.e., TP edges) against the number of false edges selected (i.e., FP edges) for the FGL model. We say that an edge $(i,j)$ in the $l$-th network is selected in the estimate $\widehat{\Theta}^{(l)}$ if $\widehat{\Theta}^{(l)}_{ij} \neq 0$, and we say that the edge is true in the precision matrix $(\Sigma^{(l)})^{-1}$ if $((\Sigma^{(l)})^{-1})_{ij} \neq 0$ and false if $((\Sigma^{(l)})^{-1})_{ij} = 0$.
We can see from the figure that the FGL model with $ \lambda_2=0.005$ can recover almost all of the true positive edges without false positive edges.
Figure~\ref{figsub-1} also  {shows} that for the FGL model the similarity control parameter $\lambda_2=0.005$ is much better than $\lambda_2=0.05$ in terms of the ability of true edges detection. When $\lambda_2=0.05$, the FGL model can merely detect about 3000 true positive edges while the the number of false positive edges is increased to over 600. One possible reason is that $\lambda_2=0.05$ is too large compared with the underlying optimal one in this case.

Figure~\ref{figsub-2} illustrates the sum of squared errors between estimated edge values and true edge values, i.e., $\sum_{l=1}^{L} \sum_{i < j}\big( \widehat{\Theta}^{(l)}_{ij} - ((\Sigma^{(l)})^{-1})_{ij}\big)^2$. When the number of the total edges selected is increasing (i.e., the sparsity control parameter is decreasing), the error is decreasing and finally reaches a fairly low value.

Figure~\ref{figsub-3}  plots the number of true positive differential edges against false positive differential edges. A differential edge is an edge that differs between classes and thus corresponds to a change-point. We say that the $(i,j)$ edge is estimated to be differential between the $l$-th and the $(l+1)$-th networks if $|\widehat{\Theta}^{(l)}_{ij} - \widehat{\Theta}^{(l+1)}_{ij} | > 10^{-6}$, and we say that it is truly differential if $| ((\Sigma^{(l)})^{-1})_{ij} - ((\Sigma^{(l+1)})^{-1})_{ij} |> 10^{-6}$.
The number of differential edges is computed for all successive pairs of networks.
The best point in Figure \ref{figsub-3} is the red one which has approximately 2700 true positive differential edges and almost no false one.
We can also see from Figure \ref{figsub-3} that  all the blue points have no false positive differential edge and small numbers of true positive differential edges. This might  be caused by the larger similarity control parameter $\lambda_2=0.05$ which
forces an excessive number of edges across $L$ networks to be similar.

\begin{figure}[!h]
\centering
\begin{subfigure}[b]{.32\linewidth}
\centering
\includegraphics[width=1.0\textwidth]{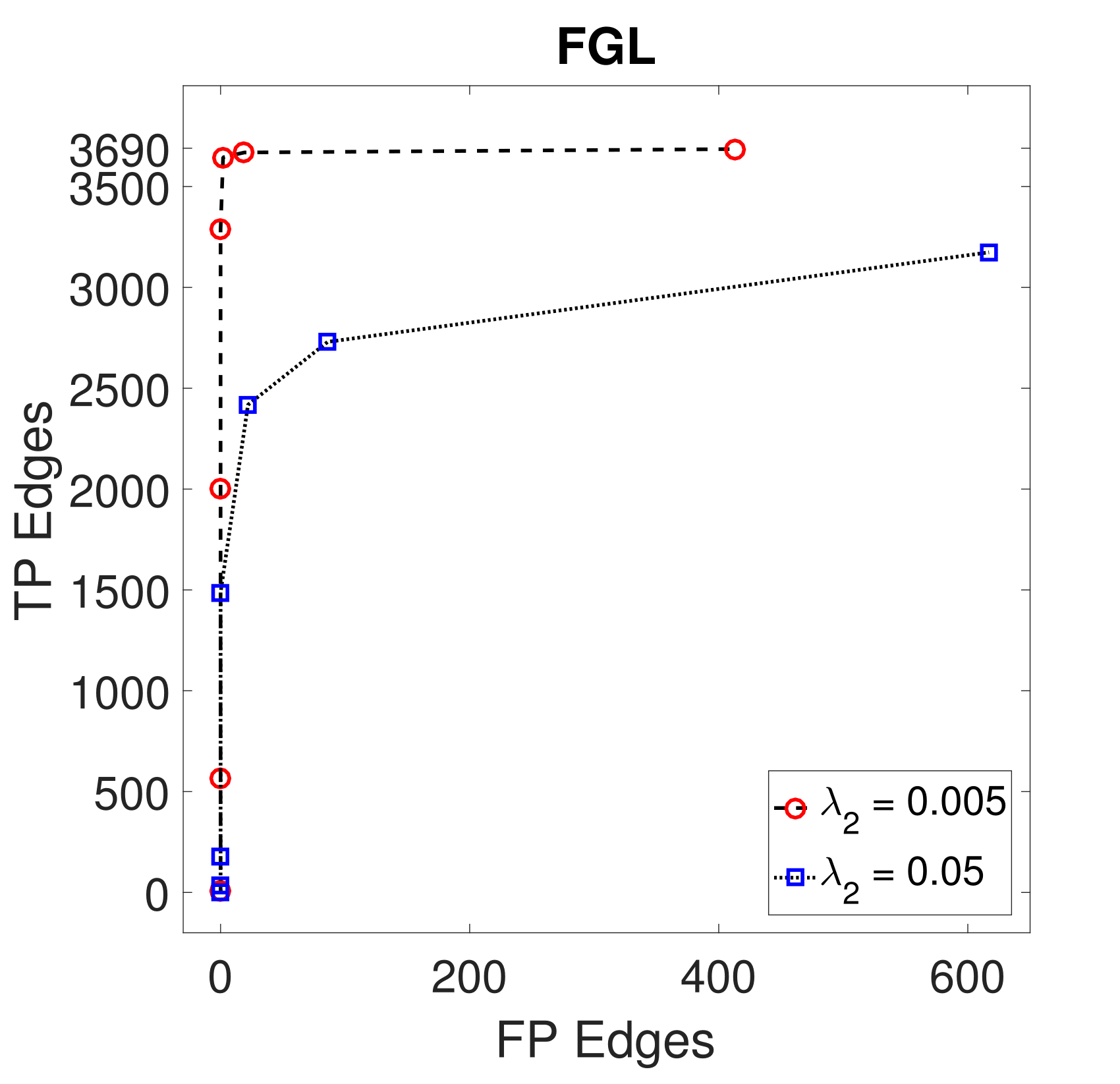}
\caption{}\label{figsub-1}
\end{subfigure}
\begin{subfigure}[b]{.32\linewidth}
\centering
\includegraphics[width=1.0\textwidth]{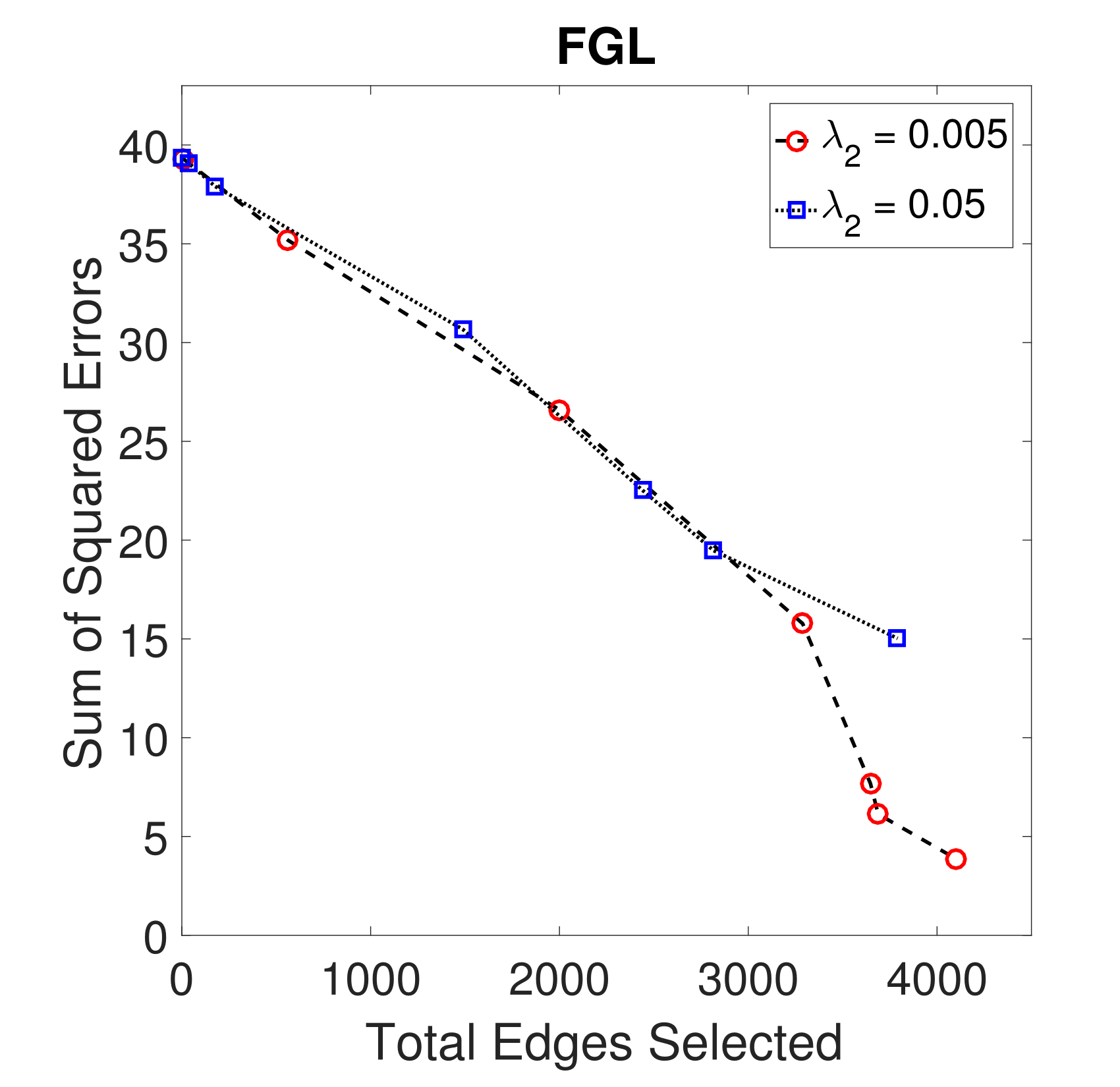}
\caption{}\label{figsub-2}
\end{subfigure}
\begin{subfigure}[b]{.32\linewidth}
\centering
\includegraphics[width=1.0\textwidth]{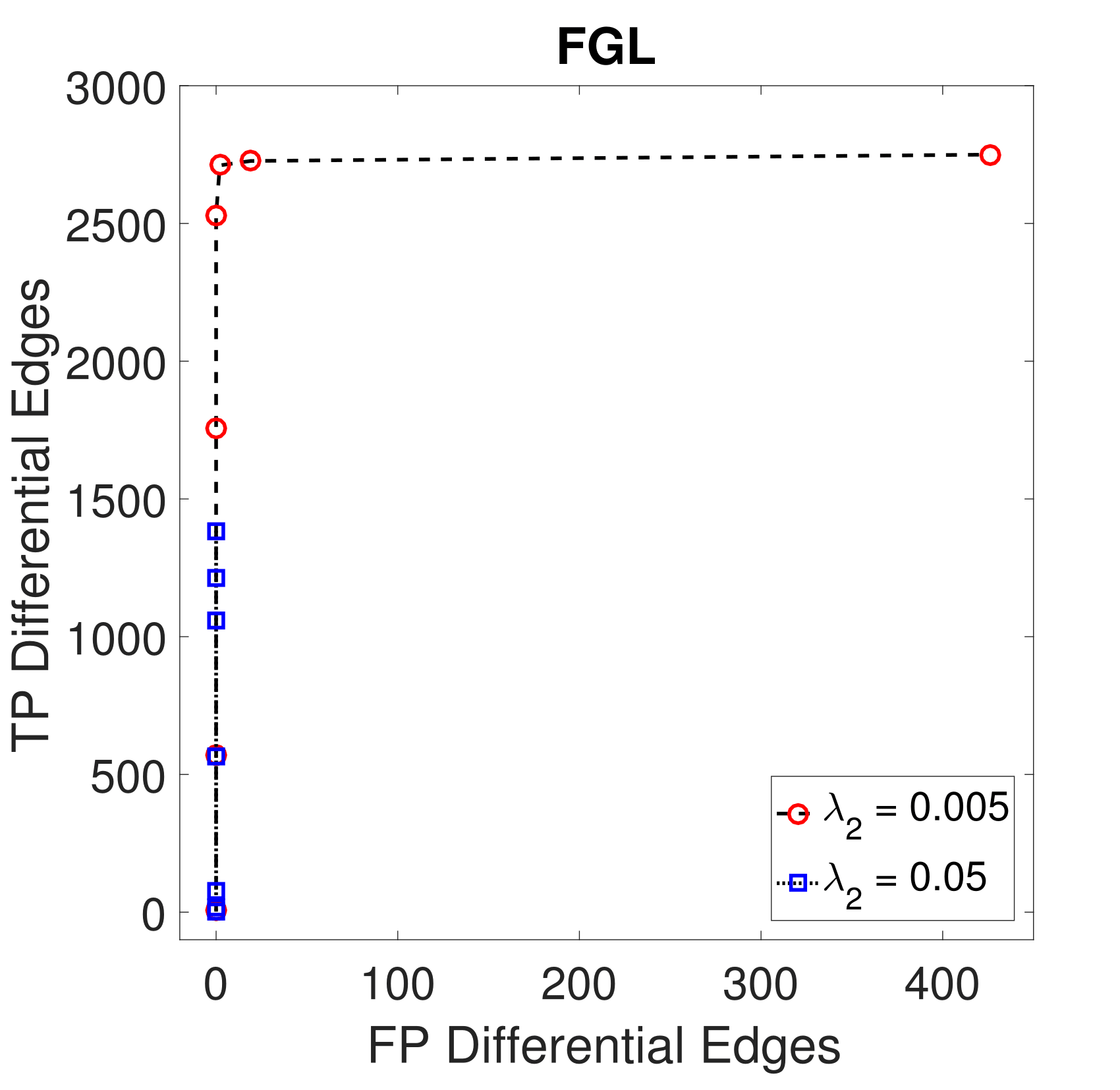}
\caption{}\label{figsub-3}
\end{subfigure}
\caption{\small{Performances on nearest-neighbour networks with $p=500$ and $L=3$.
(a) number of edges correctly identified to be nonzero (true positive edges) versus number of edges incorrectly identified to be nonzero (false positive edges); (b) sum of squared errors in edge values versus the total number of edges estimated to be nonzero; (c) number of edges correctly found to have values differing between successive classes (true positive differential edges) versus  number of edges incorrectly found to have values differing between successive classes (false positive differential edges).
}}
 \label{fig-nn}
\end{figure}

\subsection{Standard \& Poor's 500 Stocks}
In this section, we compare rPPA, ADMM, and MGL on the {\tt Standard \& Poor's 500 stock price} data sets. The stock price data sets contain daily returns of 500 stocks  over a long period, and can be downloaded from the link \url{www.yahoo.com}. The dependency structures of different stocks vary over time. But it appears that the dependency networks change smoothly  over time. Therefore, the FGL model might be able to find the interactions among these stocks and how they evolve over time.

We first consider a relatively short three-year time period from January 2004 to December 2006. During this period, there are totally 755  daily returns of 370 stocks. We call this data set {\it{SPX3a}}.
For each year, it contains approximately 250  daily returns of each stock.
Considering the limited number of observations in each year and the interpretation of the results, we choose to analyse random smaller subsets of all involved stocks, whose sizes are chosen to be $p=100$ and $p=200$, over $L=3$ periods.

In addition to the above data set over three years, a relatively long  period from January 2004 to December 2014 is also considered in the experiments, which is referred to as {\it {SPX11b}}. Since the time period is longer than the previous one, the number of stocks becomes smaller as some stocks might disappear. During the 11-year time period, there are 2769 daily closing prices of 272 stocks. We can set a relatively large parameter $L=11$ according to years from January 2004 to December 2014.
Again, we choose to analyse two random subsets of all existing stocks, of which the sizes are selected to be $p=100$ and $p=200$.
\begin{table}[!ht]\centering
\caption{\small{Performances of rPPA, ADMM, and MGL on stock price data. Tolerance $\varepsilon=$ 1e-6.}}\label{table-spx}
{\scriptsize\begin{tabular}{lllccccccccc}
\toprule
Problem & $(\lambda_1,\lambda_2)$ & Density & \multicolumn{3}{c}{Iteration} & \multicolumn{3}{c}{Time} & \multicolumn{3}{c}{Error} \\
\cmidrule(l){4-6} \cmidrule(l){7-9} \cmidrule(l){10-12}
 $(p,L)$ & & & P & A & M & P & A & M & P & A & M\\
\midrule
 & (1e-04,1e-05) & 0.039 & 25 & 3701 & 6 & {\color{red} 06} & 33 & 33 & 6.0e-07 & 9.5e-07 & 1.1e-06 \\
{\it SPX3a} & (5e-05,5e-06) & 0.144 & 24 & 3701 & 9 & {\color{red} 08} & 33 & 43 & 9.9e-07 & 9.5e-07 & 4.0e-06 \\
(100,3) & (2e-05,2e-06) & 0.241 & 26 & 5359 & 19 & 11 & 49 & {\color{red} 09} & 7.0e-07 & 1.0e-06 & 1.2e-05 \\
&&&&&&&&&&&\\[-0.1cm]
 & (1e-04,1e-05) & 0.025 & 24 & 3301 & 9 & {\color{red} 15} & 01:14 & 01:26 & 8.1e-07 & 8.2e-07 & 2.3e-06 \\
{\it SPX3a} & (5e-05,5e-06) & 0.086 & 24 & 3301 & 17 & {\color{red} 22} & 01:15 & 03:22 & 6.8e-07 & 8.3e-07 & 4.7e-06 \\
(200,3) & (2e-05,2e-06) & 0.150 & 26 & 5920 & 44 & {\color{red} 30} & 02:17 & 03:59 & 6.4e-07 & 1.0e-06 & 8.0e-06 \\
&&&&&&&&&&&\\[-0.1cm]
 & (5e-04,5e-05) & 0.028 & 24 & 3701 & 8 & {\color{red} 18} & 01:23 & 02:40 & 9.6e-07 & 1.0e-06 & 3.2e-06 \\
{\it SPX11b} & (1e-04,1e-05) & 0.126 & 24 & 3701 & 111 & {\color{red} 27} & 01:55 & 04:59 & 9.3e-07 & 1.0e-06 & 4.9e-06 \\
(100,11) & (5e-05,5e-06) & 0.206 & 24 & 3710 & 388 & {\color{red} 30} & 01:59 & 13:16 & 9.4e-07 & 1.0e-06 & 5.6e-06 \\
&&&&&&&&&&&\\[-0.1cm]
 & (5e-04,5e-05) & 0.017 & 22 & 3501 & 28 & {\color{red} 54} & 03:58 & 46:47 & 6.2e-07 & 9.9e-07 & 1.7e-06 \\
{\it SPX11b} & (1e-04,1e-05) & 0.081 & 24 & 3601 & 477 & {\color{red} 01:19} & 05:13 & 01:34:40 & 9.5e-07 & 8.8e-07 & 3.5e-06 \\
(200,11) & (5e-05,5e-06) & 0.134 & 24 & 3573 & 1076 & {\color{red} 01:38} & 05:21 & 03:00:00 & 9.8e-07 & 1.0e-06 & 4.0e-05 \\
\bottomrule
\end{tabular}}
\end{table}

Table~\ref{table-spx} shows the comparison of rPPA, ADMM, and MGL on the stock price data sets {\it{SPX3a}} and {\it{SPX11b}} with $100$ and $200$ selected stocks.
One outstanding observation from the table is that rPPA outperforms ADMM and MGL by an obvious margin and rPPA is faster than the other two methods except for one instance.
For the exceptional instance, rPPA is still faster than ADMM and merely two seconds slower than MGL.
In addition, we find that both rPPA and ADMM succeeded in solving all instances; while MGL failed to solve two of them within one hour. This might imply that MGL is not robust for solving the FGL model
when applied to the stock price data sets. The numerical results show convincingly that our algorithm rPPA can solve the FGL problems  efficiently and robustly. The superior performance of rPPA can mainly be attributed to our ability to extract and exploit the sparsity structure (in the surrogate generalized Jacobian of ${\rm Prox}_{\mathcal{P}}$)
within the SSN method to solve each rPPA subproblem very efficiently.
\begin{figure}[!ht]
\centering
\includegraphics[width=1.0\textwidth]{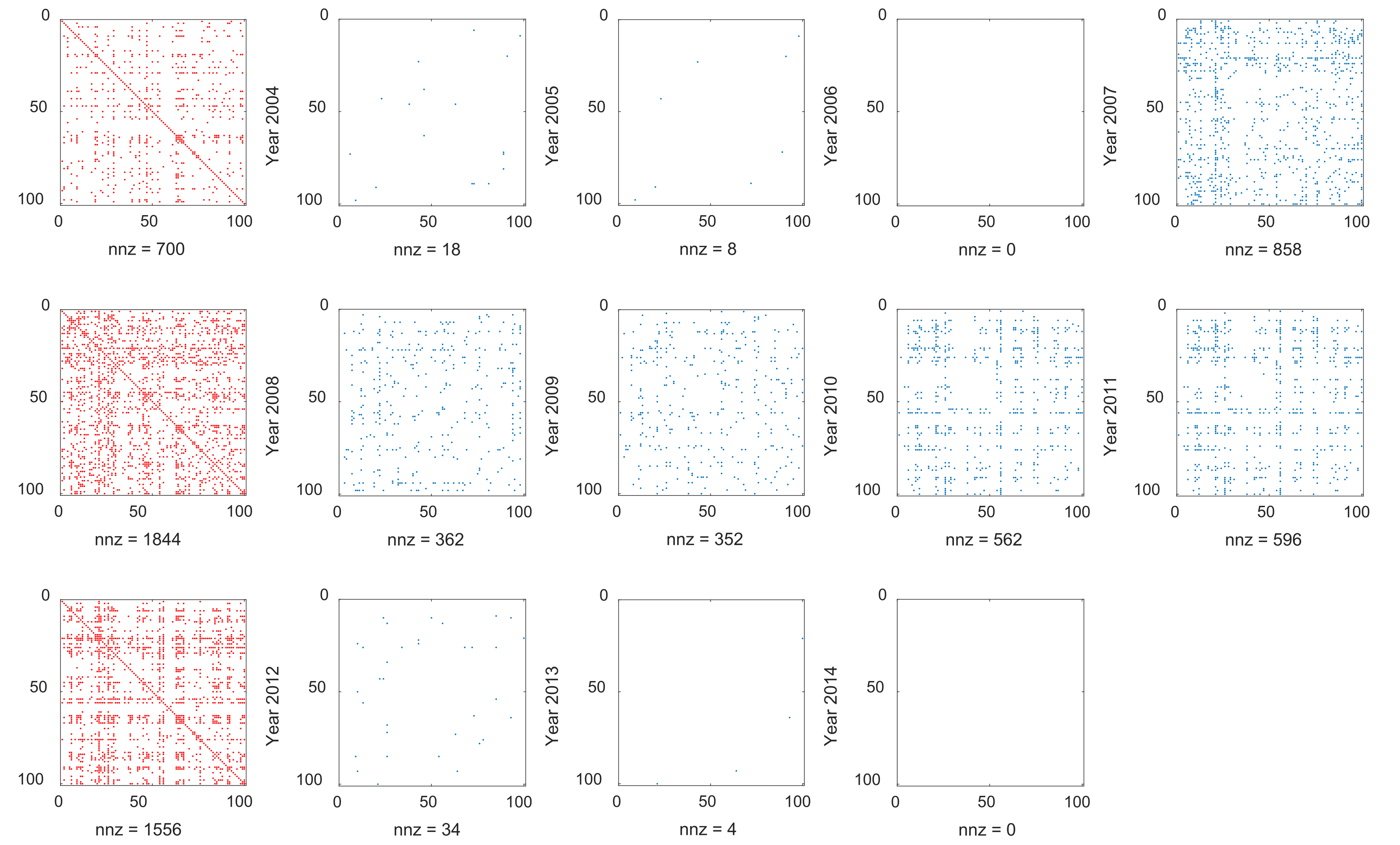}
\caption{\small{Patterns of the estimated precision matrices over 11 years on the stock price data sets. The red pattern extracts the common structure of those in the same row. The blue pattern corresponds to individual edges specific to its own network.}}\label{fig-spx11}
\end{figure}

Figure~\ref{fig-spx11} displays the sparse patterns of $11$ estimated precision matrices from year 2004 to year 2014 on the data set {\it SPX11b} with $(\lambda_1,\lambda_2) =$ (1e-4,1e-4) in application of the FGL model.
It should be noted that this period covers the 2008 financial crisis.
We manually split the time points into three stages (one stage corresponds to one row in Figure \ref{fig-spx11}) to aid the interpretation of the results.
Each  red pattern in the left panel presents the common structure across the estimated precision matrices in its stage.
And each blue pattern visualizes the individual edges specific to its own precision matrix. Generally, one can hardly expect a meaningful common structure across all the $11$ time points,  and thus we provide here the common structure across parts of nearby precision matrices.
One can clearly see that more edges are detected in the middle stage, and the number of the common edges across year  2008, 2009, 2010, 2011 is correspondingly larger than that in the earlier and later stages. The increased amount of interactions among the stocks over this period is likely due to the 2008 global financial crisis and its sustained effects.
Another observation is that the number of edges had a drastic increase in 2007, kept at a high level during the 2008 global financial crisis and a  certain period  after that, and then went down to a  level still higher than that of the pre-crisis period (year 2004, 2005, 2006). The sudden increase in 2007 might be  seen as a prediction of the oncoming financial crisis in 2008. The increased amount of interactions among stocks after the financial crisis compared to that in the pre-crisis period may indicate some essential changes of the financial landscape. To a certain degree, the observations agree well with those observed by \cite{yang2018estimating}.

\subsection{University Webpages}
Here we evaluate  the numerical performances of rPPA, ADMM, and MGL on the data set {\tt university webpages}, which is provided by \cite{ana2007improving} and available at
\url{http://ana.cachopo.org/datasets-for-single-label-text-categorization.}
The original pages were collected from  {computer science departments of} various universities in 1997. We selected four largest and meaningful classes in our experiment: Student, Faculty, Staff, and Department. For each class, the collection contains pages from four universities: Cornell, Texas, Washington, Wisconsin, and other miscellaneous pages from other universities.
Furthermore, the original text data have been preprocessed by stemming techniques, that is, reducing words to their morphological roots.
The preprocessed data sets downloaded from the link above contain two files: two thirds of the pages were randomly chosen as training set ({\it{Webtrain}}) and the remaining third as testing set ({\it{Webtest}}). Table \ref{table-doc} presents the distribution of documents per class.
\begin{table}[!ht]\centering
\caption{\small{Distribution of documents of classes Student, Faculty, Course, and Project.}}\label{table-doc}
{\scriptsize\begin{tabular}{lccccc}
\toprule
Class & Student & Faculty & Course & Project & Total\\
\midrule
\#train docs& 1097& 750 & 620 & 336 & 2803\\
\#test docs & 544 & 374 & 310 & 168 & 1396\\
\bottomrule
\end{tabular}}
\end{table}

Next, we apply the FGL model  to the ${\it Webtest}$ data set for the purpose of interpreting the  data. We choose tuning parameters that enforce high sparsity and similarity. In our experiment, we set $\lambda_1=0.005$ and $\lambda_2=0.003$. The resulting common structure is displayed in Figure~\ref{fig-webpage-common}. The thickness of an edge is proportional to the magnitude of the associated average partial correlation. Figure~\ref{fig-webpage-common} shows that some standard phrases in computer science, such as program-languag, oper-system, distribut-system, softwar-engin, possess high partial correlations among their constituent words in all four classes. It successfully demonstrates the effectiveness of the FGL model for exploring the similarity across related classes. Table~\ref{table-web}  shows the comparison of the three methods rPPA, ADMM, and MGL on the webpages data sets with data dimension $p=100$, $p=200$, and $p=300$. As can be seen, rPPA outperforms ADMM and MGL by a large margin for  most of the tested webpages data sets.
\begin{figure}[!ht]
 \centering
 \includegraphics[width=0.8\textwidth]{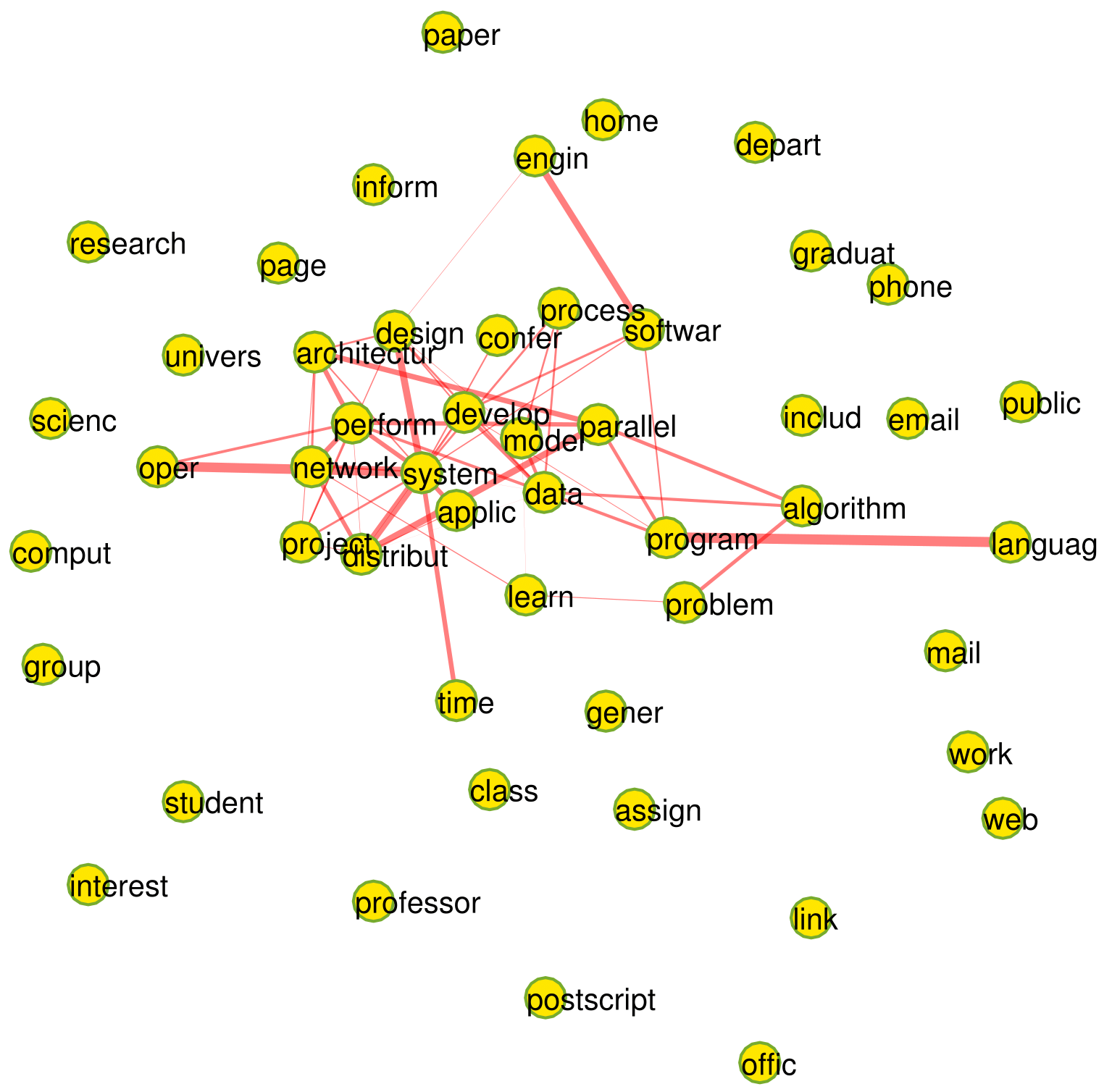}
 \caption{\small{Common structure in {\it Webtest} data. The nodes  {represent}  $50$ words with highest frequencies. The width of an edge is proportional to the average magnitude of the partial correlation.}}\label{fig-webpage-common}
\end{figure}
\begin{table} [!ht]\centering
\caption{\small{Performances of rPPA, ADMM, and MGL on webpages data. Tolerance $\varepsilon =$ 1e-6.}}\label{table-web}
{\scriptsize\begin{tabular}{lllccccccccc}
\toprule
Problem & $(\lambda_1,\lambda_2)$ & Density & \multicolumn{3}{c}{Iteration} & \multicolumn{3}{c}{Time} & \multicolumn{3}{c}{Error} \\
\cmidrule(l){4-6} \cmidrule(l){7-9} \cmidrule(l){10-12}
 $(n,L)$ & & & P & A & M & P & A & M & P & A & M\\
\midrule
 & (1e-02,1e-03) & 0.015 & 16 & 2401 & 4 & 06 & 25 & {\color{red} 05} & 5.8e-07 & 9.9e-07 & 1.2e-07 \\
 {\it Webtest} & (5e-03,5e-04) & 0.047 & 16 & 2401 & 6 & {\color{red} 07} & 27 & 10 & 5.6e-07 & 9.9e-07 & 3.1e-07 \\
 (100,4) & (1e-03,1e-04) & 0.219 & 15 & 701 & 38 & {\color{red} 05} & 09 & 49 & 5.7e-07 & 6.1e-07 & 9.0e-07 \\
&&&&&&&&&&&\\[-0.1cm]
 & (1e-02,1e-03) & 0.008 & 18 & 2101 & 7 & {\color{red} 11} & 01:03 & 59 & 7.3e-07 & 8.5e-07 & 8.6e-07 \\
 {\it Webtest} & (5e-03,5e-04) & 0.025 & 18 & 2101 & 8 & {\color{red} 12} & 01:04 & 01:05 & 6.8e-07 & 8.5e-07 & 4.7e-07 \\
 (200,4) & (1e-03,1e-04) & 0.156 & 18 & 2101 & 72 & {\color{red} 12} & 01:12 & 07:31 & 5.5e-07 & 7.1e-07 & 9.3e-07 \\
&&&&&&&&&&&\\[-0.1cm]
 & (5e-03,5e-04) & 0.016 & 18 & 2101 & 9 & {\color{red} 25} & 02:12 & 03:44 & 5.9e-07 & 8.3e-07 & 3.7e-07 \\
 {\it Webtest} & (1e-03,1e-04) & 0.119 & 17 & 2101 & 258 & {\color{red} 49} & 02:16 & 39:58 & 7.9e-07 & 8.6e-07 & 1.0e-06 \\
 (300,4) & (5e-04,5e-05) & 0.244 & 19 & 2901 & 1393 & {\color{red} 01:18} & 03:07 & 02:22:23 & 5.6e-07 & 8.0e-07 & 1.0e-06 \\
&&&&&&&&&&&\\[-0.1cm]
 & (1e-02,1e-03) & 0.011 & 24 & 20000 & 3 & 12 & 03:11 & {\color{red} 04} & 8.4e-07 & 7.0e-06 & 3.4e-06 \\
 {\it Webtrain} & (5e-03,5e-04) & 0.030 & 24 & 20000 & 5 & 13 & 03:35 & {\color{red} 08} & 8.1e-07 & 7.0e-06 & 8.9e-07 \\
 (100,4) & (1e-03,1e-04) & 0.162 & 24 & 20000 & 22 & {\color{red} 14} & 03:52 & 01:06 & 7.3e-07 & 7.0e-06 & 3.5e-06 \\
&&&&&&&&&&&\\[-0.1cm]
 & (5e-03,5e-04) & 0.015 & 24 & 20000 & 5 & 46 & 10:53 & {\color{red} 17} & 8.5e-07 & 6.7e-06 & 1.2e-06 \\
 {\it Webtrain} & (1e-03,1e-04) & 0.105 & 24 & 15227 & 33 & {\color{red} 44} & 08:23 & 03:02 & 7.2e-07 & 1.0e-06 & 2.6e-06 \\
 (200,4) & (5e-04,5e-05) & 0.210 & 24 & 20000 & 95 & {\color{red} 52} & 10:41 & 06:34 & 6.9e-07 & 6.6e-06 & 2.7e-06 \\
&&&&&&&&&&&\\[-0.1cm]
 & (5e-03,5e-04) & 0.010 & 24 & 20000 & 7 & {\color{red} 01:31} & 21:31 & 02:07 & 8.1e-07 & 5.9e-06 & -1.2e-08 \\
 {\it Webtrain} & (1e-03,1e-04) & 0.077 & 24 & 20000 & 52 & {\color{red} 01:33} & 22:06 & 20:47 & 6.3e-07 & 5.9e-06 & 1.8e-06 \\
 (300,4) & (5e-04,5e-05) & 0.168 & 24 & 20000 & 155 & {\color{red} 01:51} & 21:56 & 18:09 & 6.1e-07 & 5.9e-06 & 1.8e-06 \\
\bottomrule
\end{tabular}}
\end{table}

\section{Conclusion}
We have designed an efficient and globally convergent regularized proximal point algorithm for solving the primal formulation of the fused graphical Lasso problem. From a theoretical perspective, we established the Lipschitiz continuity of the solution mapping and consequently obtained that the
primal and dual sequences are locally linearly convergent. This lays the foundation for the efficiency of the proposed algorithm. Moreover, the second order information was also fully exploited,
which further leads to the high efficiency of the proposed algorithm.
Numerically, we demonstrated the superior efficiency and robust performance of the proposed method by comparing it with the extensively used alternating direction method of multipliers and the proximal Newton-type method \citep{yang2015fused} on both  synthetic and real data sets. In summary, the proposed semismooth Newton based  regularized proximal point algorithm is a highly efficient method for solving the fused graphical Lasso problems.

\section{Supplementary Materials}

\begin{description}
  \item[Supplementary material:] It contains technical details (generalized Jacobian of the proximal mapping of the fused Lasso regularizer and implementation of ADMM) and  numerical results (on data  {\tt University Webpages} and {\tt 20 Newsgroups}). (pdf file)
\end{description}
\bibliographystyle{Chicago}

\clearpage
\setcounter{equation}{0}
\setcounter{table}{0}
\setcounter{figure}{0}
\setcounter{page}{1}
\begin{appendices}

\renewcommand*{\thesection}{\arabic{section}}
\begin{center}
  {\LARGE \bf Supplementary material to ``An Efficient Linearly Convergent Regularized Proximal Point Algorithm for Fused Multiple Graphical Lasso Problems''}
\end{center}
\begin{center}
\large Ning Zhang \,\,\, Yangjing Zhang \,\,\, Defeng Sun \,\,\, Kim-Chuan Toh
\end{center}
\begin{center}
\date{February 19, 2019}
\end{center}
\section{Generalized Jacobian of the Proximal Mapping of the Fused Lasso Regularizer}
In this section, we recall the characterization of the generalized Jacobian of the fused Lasso regularizer \citep{tibshirani2005sparsity}, which will be used to derive the explicit expression of the generalized Jacobian of the
fused graphical Lasso (FGL) regularizer.

The fused Lasso regularizer is defined by
$
\varphi(x)=\lambda_1\|x\|_1+\lambda_2\|Bx\|_1,\,\forall\, x\in\Re^L,
$
where the matrix $B\in \Re^{(L-1) \times L}$ is defined by $Bx = [x_1 - x_2;  \dots; x_{L-1} - x_L]$. Denote the proximal mapping of $\lambda_2\|B \cdot\|_1$ by
$
x_{\lambda_2}(v):=\arg\min_x \left\{\lambda_2\|Bx\|_1+\frac{1}{2}\|x-v\|^2\right\},\,\,\forall\,v\in\Re^L.
$

\begin{lemma}\cite[Proposition 1]{friedman2007pathwise} 
Given $\lambda_1,\lambda_2\geq 0$, it holds that
$
{\rm Prox}_{\varphi}(v)= {\rm Prox}_{\lambda_1 \|\cdot\|_1} (x_{\lambda_2}(v))=
{\rm sign}(x_{\lambda_2}(v))\odot \max(|x_{\lambda_2}(v)|-\lambda_1,0),\,\,\forall\,v\in\Re^L.
$
\end{lemma}
\begin{lemma}\cite[Lemma 1]{li2018efficiently}
Given $\lambda_2\geq 0$, it holds that
$
x_{\lambda_2}(v) = v-B^T z_{\lambda_2}(Bv),\,\,\forall\,v\in\Re^{L},
$
where $z_{\lambda_2}(u):=\arg\min\limits_{z} \left\{
\frac{1}{2}\|B^T z\|^2 - \langle z,\,u\rangle~|\,\,\|z\|_{\infty}\leq\lambda_2\right\},\,\,\forall\,u\in\Re^{L-1}.$
\end{lemma}
Given $v\in \Re^L$, consider the following sets:
$$
\begin{array}{rl}
\mathcal{I}_{z}(v)&:=\left\{i\,|\,\, |(z_{\lambda_2}(B v))_i|=\lambda_2,\,\,i=1,2,\dots,L-1\right\},\\[3pt]
\mathcal{K}_z(v)  &:=\left\{K\subseteq\{1,2,\ldots,L-1\}\,|\,{\rm supp}(B x_{\lambda_2}(v))\subseteq K\subseteq \mathcal{I}_{z}(v) \right\}.
\end{array}
$$
Define the multifunction $\mathcal{Q}_z:\Re^L \rightrightarrows \Re^{(L-1)\times (L-1)}$ by
$$
\mathcal{Q}_{z}(v):=\left\{\widehat{Q}\in\Re^{(L-1)\times(L-1)}\,|\,
\widehat{Q}=(\Sigma_K BB^T \Sigma_K)^{\dag},\,K\in\mathcal{K}_z(v) \right\},
$$
where $\Sigma_K={\rm Diag}(\sigma_K)\in\Re^{(L-1)\times(L-1)}$ with
$ (\sigma_K)_{i} = 0$, if $i\in K$, $ (\sigma_K)_{i} = 1$, if $i\notin K$.
Define also
the multifunction $\mathcal{Q}_x:\,\Re^L \rightrightarrows \Re^{L\times L}$ by
$$
\mathcal{Q}_{x}(v):=\left\{Q\in\Re^{L\times L}\,|\,
Q=I-B^T\widehat{Q}B,\,\,\widehat{Q}\in\mathcal{Q}_{z}(v)\right\}.
$$
It has been shown by \cite{li2018efficiently} that
the surrogate generalized Jacobian of ${\rm Prox}_{\varphi}$ at $v$ is the set
$$
\widehat{\partial}{\rm Prox}_{\varphi}(v):=\left\{ M\in\S_+^L\,|\,M = \Upsilon Q,\,\,\Upsilon\in\partial_B{\rm Prox}_{\lambda_1\|\cdot\|_1}(x_{\lambda_2}(v)),\,\,Q\in\mathcal{Q}_{x}(v)\right\},
$$
where $\partial_B{\rm Prox}_{\lambda_1\|\cdot\|_1} $ denotes the B-subdifferential of ${\rm Prox}_{\lambda_1\|\cdot\|_1}$ \cite[Equation (2.12)]{qi1993convergence}.

\section{Implementation of ADMM}
In this part, we briefly describe the alternating direction method of multipliers (ADMM) for solving the dual problem of the FGL problem:
\begin{equation*}
\begin{array}{cl}
\max\limits_{X} & \displaystyle\sum^L_{l=1} \left(\log \det \,X^{(l)}+ p\right)- \mathcal{P}^*(X-S).
\end{array}
\end{equation*}
This can be rewritten equivalently as follows:
\begin{equation}\label{model-ADMM}
\begin{array}{cl}
\min\limits_{X,\,Z} & \displaystyle \left\{ \sum^L_{l=1} \left({\rm -log\,det}\, X^{(l)}\right)+ \mathcal{P}^*(Z) \,\Big|\, X-Z=S\right\}.
\end{array}
\end{equation}
The augmented Lagrangian function associated with \eqref{model-ADMM}, given  $\sigma>0$, is defined by
$$
\begin{array}{l}
\widehat{\mathcal{L}}_{\sigma} (X,Z,\Theta)=\sum^L_{l=1} \left({\rm -log\,det}\, X^{(l)}\right)+ \mathcal{P}^*(Z)+\langle X-Z-S,\,\Theta\rangle+\displaystyle\frac{\sigma}{2}\|X-Z-S\|^2.
\end{array}
$$
The KKT optimality conditions are as follows:
\begin{equation}\label{def-KKT-MGL-D}
\left\{\begin{array}{l}
\Theta-{\rm Prox}_{\mathcal{P}}(\Theta+Z)=0,\\[2mm]
 X - Z - S=0,\\[2mm]
\Omega^{(l)}-{\rm Prox}_{\vartheta}(\Omega^{(l)}-X^{(l)})=0,\,\,l=1,2,\ldots,L,
\end{array}\right.
\end{equation}
where $\vartheta(X)=-\log\det\,X$ if $X\in\mathbb{S}^p_{++}$ and $\vartheta(X)=+\infty$ otherwise.
Due to its separable structure in terms of the variables $X$ and $Z$, ADMM is often considered as a natural choice for solving \eqref{model-ADMM}. The classic ADMM was first proposed by \cite{glowinski1975approximation,gabay1976dual}, and later extended by \cite{fazel2013hankel,chen2017efficient}. The iteration scheme of ADMM for \eqref{model-ADMM} can be described as follows: given $\tau\in(0,(1+\sqrt{5})/2)$, and an initial point $(X^0,Z^0,\Theta^0)$, the $(k+1)$-th iteration is given by
\begin{eqnarray}
X^{k+1}&=&\arg\min_{X}~\widehat{\mathcal{L}}_{\sigma} (X,Z^k,\Theta^k)\nonumber\\[-2mm]
Z^{k+1}&=&\arg\min_{Z}~\widehat{\mathcal{L}}_{\sigma} (X^{k+1},Z,\Theta^k)
\;=\;
(X^{k+1}+ \Theta^k/\sigma-S)-{\rm Prox}_{\mathcal{P}}(X^{k+1}+ \Theta^k/\sigma -S),
\nonumber
\\[0mm]
\Theta^{k+1}&=&\Theta^k + \tau\sigma(X^{k+1}-Z^{k+1}-S).
\nonumber
\end{eqnarray}
Here, $X^{k+1} = ((X^{(1)})^{k+1},\dots,(X^{(L)})^{k+1})$ can be updated by
$$
\begin{array}{l}
(X^{(l)})^{k+1} = \phi^+_{{\sigma}^{-1}}\left((Z^{(l)})^{k}-(\Theta^{(l)})^k/\sigma +S^{(l)} \right),\,\,l = 1,2,\ldots,L.
\end{array}
$$
In our implementation, we tune the parameter $\sigma$ wisely according to the progress of primal and dual feasibilities \cite[Section 4.4]{lam2018fast}. We also use a larger step-length $\tau$ of $1.618$, which has been demonstrated in various works to perform better than the simple case with $\tau =1$.
Based on the KKT optimality conditions \eqref{def-KKT-MGL-D},
the accuracy of an approximate optimal solution $(\Theta,X,Z)$ generated by ADMM is measured by the following
relative residual:
$$
\begin{array}{l}
\eta_A:=\max \left\{
\frac{\|\Theta-{\rm Prox}_{\mathcal{P}}(\Theta + Z)\|}{1+\|\Theta\|},\,
\frac{\|X - Z- S \|}{1+\|S \|},\,
\underset{1 \leq l \leq L}{\max}\Big\{ \frac{\|\Theta^{(l)} X^{(l)}-I\|}{1+\sqrt{p}} \Big\}
\right\}.
\end{array}
$$
Thus far, we have provided an easily implementable framework of ADMM for which each iteration requires
the computation of the proximal mapping of the log-determinant function and that of the FGL regularizer.

\section{Numerical Experiment: University Webpages}\label{sec-web}
This section presents the procedure of processing the {\tt university webpages} data
(available at \url{http://ana.cachopo.org/datasets-for-single-label-text-categorization}) and generating sample covariance matrices which is similar to the process used by \cite{guo2011joint}. Actually, the previous work \citep{guo2011joint} on  the data set {\it Webtest} applied a different penalty term to estimate multiple graphical models jointly.
For given integer $p$, the sample covariance matrices $S^{(l)},\,l=1,2,3,4$ were constructed from the data set {\it{Webtest}} in the following way:

\begin{itemize}[topsep=1pt,itemsep=-.6ex,partopsep=1ex,parsep=1ex,leftmargin=4ex]
\item[(i)] Choose $p$ words with highest frequency which appear in each class
at least once. Namely, the words we analyse are a subset of all involved words.

\item[(ii)] Obtain $X^{(1)}\in\Re^{544\times p}$ from class Student, where the $(i,j)$-th element $X^{(1)}_{ij}$ denotes the number of times the $j$-th term appears in the $i$-th page of class Student. In the same way, $X^{(2)}\in\Re^{374\times p}$, $X^{(3)}\in\Re^{310\times p}$, and $X^{(4)}\in\Re^{168\times p}$ can be obtained from class Faculty, Course, and Project, respectively. Denote their vertical concatenation by a new matrix $ X = [X^{(1)}; X^{(2)}; X^{(3)}; X^{(4)}]
\in \Re^{1396\times p}$.

\item[(iii)] The matrix $P$ is obtained by normalizing $X$ along each column: $P_{ij} = X_{ij}/\sum_i X_{ij}$. Then, the log-entropy weight of the $j$-th word is defined as $e_j = 1 + \sum_i P_{ij}(\ln P_{ij})/\ln 1396$.

\item[(iv)] Compute $\overline{X}$ as follows: $\overline{X}_{ij} = e_j \ln(1+X_{ij})$, and split $\overline{X}$ by columns accordingly: $ \overline{X} = [\overline{X}^{(1)}; \overline{X}^{(2)}; \overline{X}^{(3)}; \overline{X}^{(4)}]$.

\item[(v)] Generate sample covariance matrices $S^{(l)}$ from $\overline{X}^{(l)}$: $S^{(1)} = {\rm cov}(\overline{X}^{(1)}),\,l=1,2,3,4$.
\end{itemize}
Following the procedure described above, we can also generate sample covariance matrices from the data set {\it Webtrain}.

In section 4.3 of the paper, we have successfully demonstrated the effectiveness of the FGL model for exploring the similarity across related classes (Figure 3). On the other hand, we believe that the model can also detect the heterogeneity among different classes. As an example, supplementary Figures~\ref{figsub-course} and \ref{figsub-project} illustrate the differences between the Course (Figure~\ref{figsub-course}) and Project (Figure~\ref{figsub-project})  classes. One can see that some course related terms, {such as} class and assign, are of high degree in supplementary Figure~\ref{figsub-course}; whereas they are not even  connected in supplementary Figure~\ref{figsub-project}.
Besides, some teaching related terms are linked only in the Course category, such as class-assign, assign-problem, class-project. Overall, it is likely that the FGL model is capable of identifying the common and individual  structures of the webpages among related classes.
\begin{figure}[H]
\centering
\includegraphics[width=0.60\textwidth]{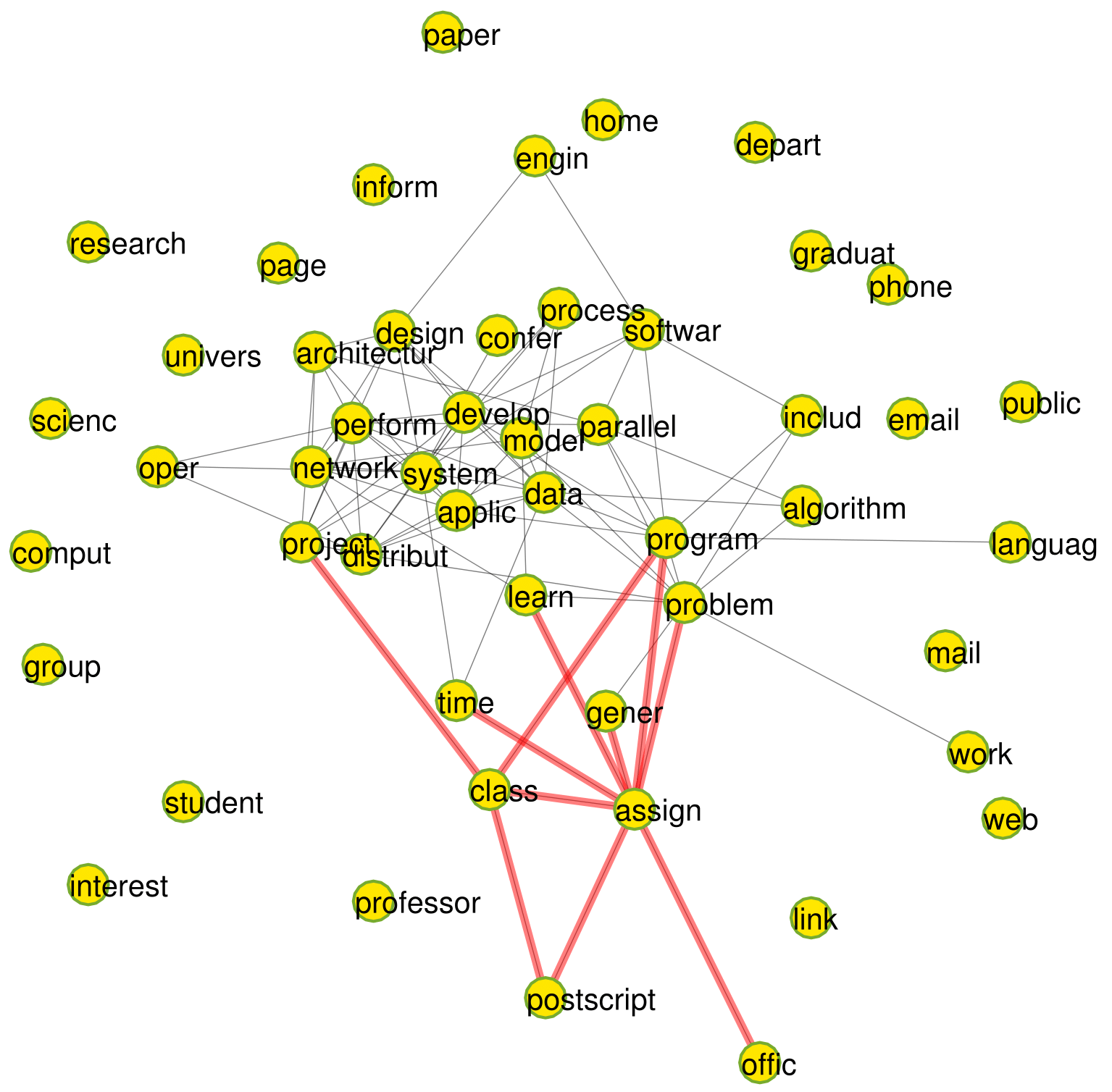}
  \caption{\small{Dependency structure for class Course. The thin black lines are the edges appearing in both classes, and the thick red lines are the edges only appearing in one class.}}
  \label{figsub-course}
\end{figure}

Supplementary Figure~\ref{Fig-webFused} presents the performance profiles of rPPA, ADMM, and MGL  for all 18 tested problems.  The meaning of the performance profiles is given as follows: a point $(x,y)$ is on the performance curve of a particular method if and only if this method can solve up to desired accuracy $(100y)\%$ of all the tested instances within at most $x$ times of the fastest
method for each instance. As can be seen,  rPPA outperforms ADMM and MGL by a large margin for the most of the tested webpages data sets. In particular, focusing on $y = 40\%$, we can see  that rPPA is around $3\sim 5$ times faster in comparison with ADMM and MGL for over $60\%$ of the tested instances.
\begin{figure}[H]
\centering
\includegraphics[width=0.60\textwidth]{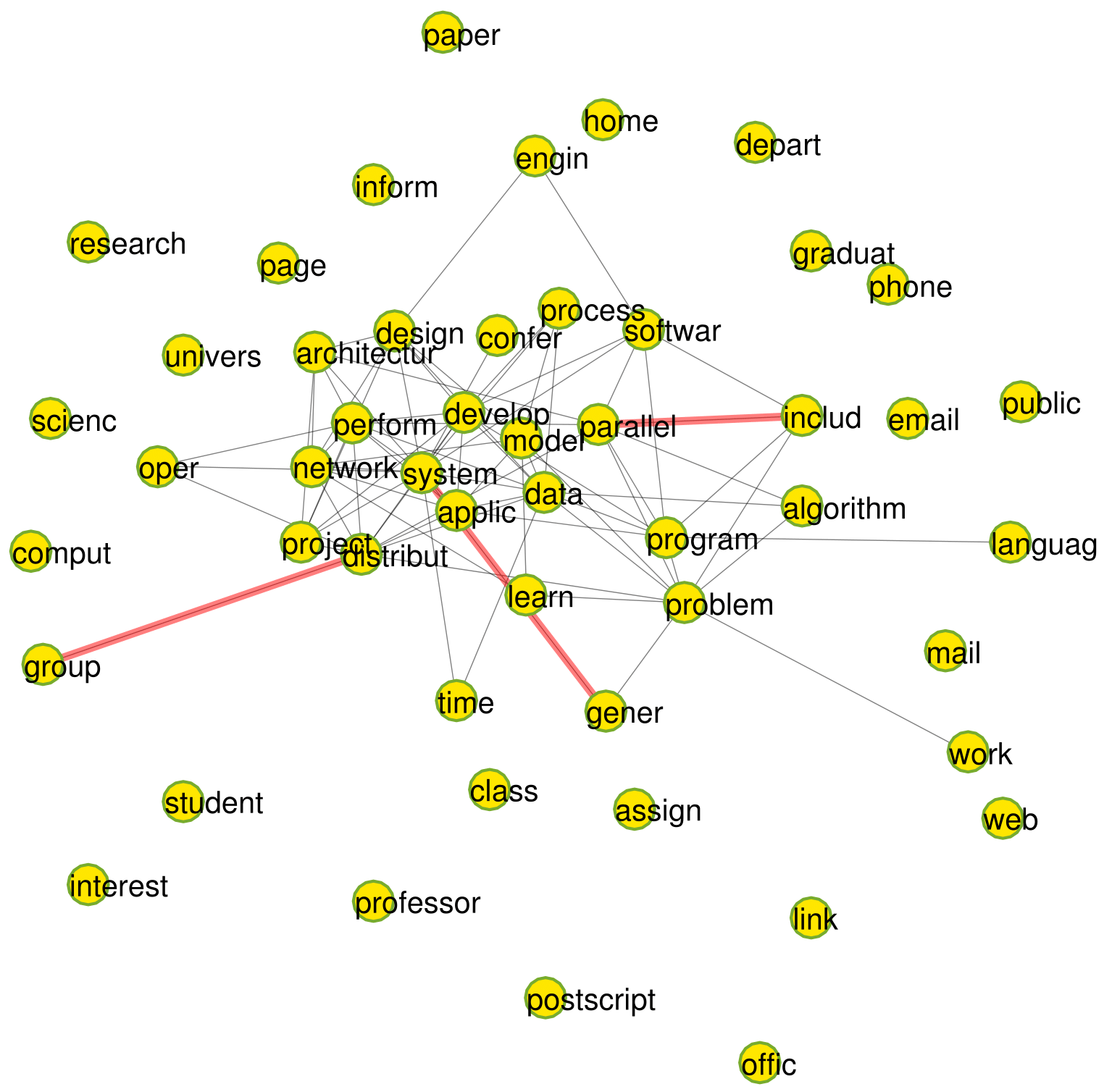}
  \caption{\small{Dependency structure for class Project. The thin black lines are the edges appearing in both classes, and the thick red lines are the edges only appearing in one class.}}
  \label{figsub-project}
\end{figure}
\begin{figure}[H]
	\centering
\includegraphics[width=0.5\textwidth]{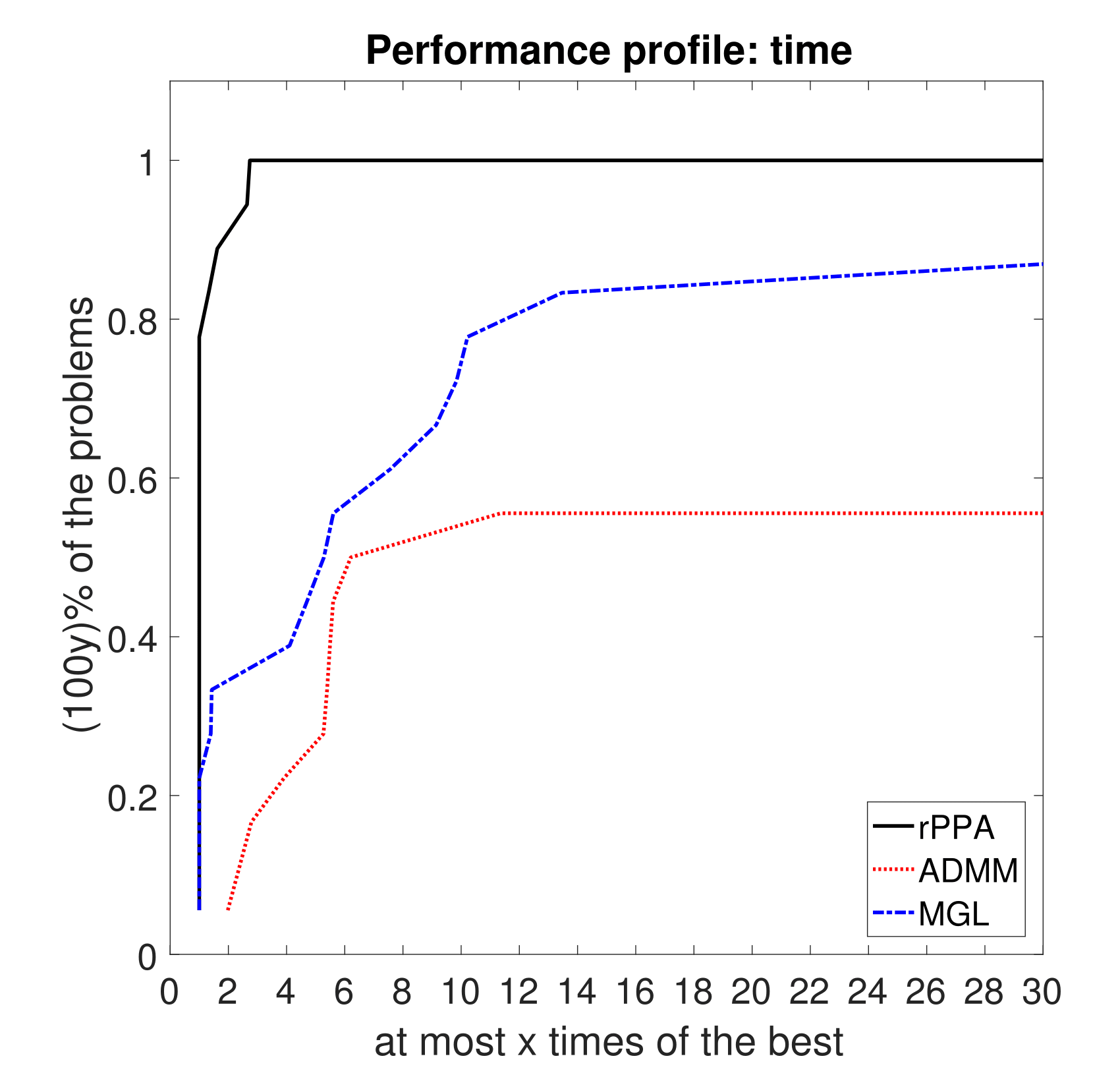}
\caption{\small{Performance profiles of the rPPA, ADMM, and MGL on university webpages data sets.}\label{Fig-webFused}}
\end{figure}

\section{Numerical Experiment: 20 Newsgroups}
This section compares rPPA, ADMM, and MGL on newsgroups data set, which is a popular text data set. The {\tt 20 newsgroups} data set is a collection of newsgroup documents, partitioned nearly evenly across $20$ different newsgroups. Different newsgroups correspond to different topics, and some of the newsgroups are closely related to each other (e.g., comp.sys.ibm.pc.hardware/comp.sys.mac.hardware), while others are highly unrelated (e.g., misc.forsale/soc.religion.christian). Supplementary Table~\ref{table-20ng} lists the 20 newsgroups (from the website:  {\tt{http://qwone.com/\~{}jason/20Newsgroups/}}), partitioned according to subject matter. According to the topics, our numerical experiments were conducted on the four subgroups, which are likely to possess common semantic structures. The four subgroups are highlighted in supplementary Table~\ref{table-20ng} and named as {\it NGcomp}, {\it NGrec}, {\it NGsci}, and {\it NGtalk} accordingly.
There are several classes in each subgroup, and we apply the FGL model to estimating jointly the precision matrices of different classes in each subgroup.
\begin{table}[!ht]\centering
\caption{\small{Partition of 20 newsgroups by topics}}\label{table-20ng}
{\scriptsize\begin{tabular}{|l|l|l|}
 \hline
 \bf{comp.graphics}             & \bf{rec.autos}          & \bf{sci.crypt} \\
 \bf{comp.os.ms-windows.misc}   & \bf{rec.motorcycles}    & \bf{sci.electronics} \\
 \bf{comp.sys.ibm.pc.hardware}  & \bf{rec.sport.baseball} & \bf{sci.med} \\
 \bf{comp.sys.mac.hardware}     & \bf{rec.sport.hockey}	  & \bf{sci.space} \\
 \bf{comp.windows.x} & & \\
 \hline
              & \bf{talk.politics.misc}    & talk.religion.misc \\
 misc.forsale & \bf{talk.politics.guns}    & alt.atheism \\
              & \bf{talk.politics.mideast} & soc.religion.christian \\
 \hline
\end{tabular}}
\end{table}

A processed version of the 20 newsgroups data set which is easy to read into {\sc Matlab} can be downloaded from Jason's page {\tt{http://qwone.com/\~{}jason/20Newsgroups/}}, and the downloaded data contains a training data set and a testing data set. We also adopted the procedure of generating sample covariance matrices described in the previous section \ref{sec-web} with a series of problem dimensionality $p=100$, $p=200$, and $p=300$.
\begin{table} [!ht]\centering
\caption{\small{Performances of rPPA, ADMM, and MGL  on  newsgroups data.  Tolerance $\varepsilon =$ 1e-6.}}\label{table-ng300}
{\scriptsize\begin{tabular}{lllccccccccc}
\toprule
Problem & $(\lambda_1,\lambda_2)$ & Density & \multicolumn{3}{c}{Iteration} & \multicolumn{3}{c}{Time} & \multicolumn{3}{c}{Error} \\
\cmidrule(l){4-6} \cmidrule(l){7-9} \cmidrule(l){10-12}
 $(n,L)$ & &  & P & A & M & P & A & M & P & A & M\\
\midrule
 {\it NGcomp}          &          (5e-03,5e-04)  & 0.020  & 19  & 4201  & 35  & {\color{red} 53}  & 05:27  & 41:26  & 5.9e-07  & 9.2e-07  & 6.2e-07 \\
 test                  &          (1e-03,1e-04)  & 0.094  & 16  & 1543  & 720  & {\color{red} 01:20}  & 02:08  & 01:54:03  & 8.8e-07  & 1.0e-06  & 1.0e-06 \\
 (300,5)               &          (5e-04,5e-05)  & 0.194  & 17  & 1391  & 1240  & {\color{red} 59}  & 01:56  & 03:00:00  & 7.9e-07  & 1.0e-06  & 9.6e-06 \\
&&&&&&&&&&&\\[-0.2cm]
 {\it NGrec}           &          (5e-03,5e-04)  & 0.004  & 25  & 20000  & 4  & {\color{red} 01:16}  & 21:18  & 04:43  & 7.1e-07  & 8.4e-06  & 8.3e-07 \\
 test                  &          (1e-03,1e-04)  & 0.061  & 25  & 20000  & 13  & {\color{red} 01:20}  & 22:10  & 04:38  & 6.3e-07  & 8.4e-06  & 4.6e-07 \\
 (300,4)               &          (5e-04,5e-05)  & 0.134  & 24  & 20000  & 37  & {\color{red} 01:23}  & 22:05  & 07:55  & 9.8e-07  & 8.4e-06  & 1.9e-06 \\
&&&&&&&&&&&\\[-0.2cm]
 {\it NGsci}           &          (5e-03,5e-04)  & 0.006  & 22  & 16244  & 6  & {\color{red} 57}  & 15:42  & 06:00  & 6.0e-07  & 1.0e-06  & 2.7e-07 \\
 test                  &          (1e-03,1e-04)  & 0.074  & 21  & 16230  & 25  & {\color{red} 01:06}  & 17:54  & 10:22  & 7.9e-07  & 1.0e-06  & 1.3e-06 \\
 (300,4)               &          (5e-04,5e-05)  & 0.156  & 21  & 16230  & 100  & {\color{red} 01:06}  & 17:52  & 20:37  & 7.3e-07  & 1.0e-06  & 1.5e-06 \\
&&&&&&&&&&&\\[-0.2cm]
 {\it NGtalk}          &          (5e-03,5e-04)  & 0.026  & 17  & 4179  & 14  & {\color{red} 52}  & 04:23  & 14:33  & 7.4e-07  & 1.0e-06  & 2.7e-07 \\
 test                  &          (1e-03,1e-04)  & 0.111  & 17  & 1018  & 79  & {\color{red} 35}  & 01:08  & 16:52  & 4.9e-07  & 1.0e-06  & 9.8e-07 \\
 (300,3)               &          (5e-04,5e-05)  & 0.228  & 17  & 922  & 434  & {\color{red} 37}  & 01:00  & 52:12  & 6.8e-07  & 1.0e-06  & 9.9e-07 \\
&&&&&&&&&&&\\[-0.2cm]
 {\it NGcomp}          &          (5e-03,5e-04)  & 0.016  & 20  & 6023  & 13  & {\color{red} 36}  & 07:36  & 26:32  & 8.6e-07  & 1.0e-06  & 5.2e-07 \\
 train                 &          (1e-03,1e-04)  & 0.077  & 20  & 6393  & 153  & {\color{red} 40}  & 08:59  & 36:18  & 6.0e-07  & 1.0e-06  & 9.7e-07 \\
 (300,5)               &          (5e-04,5e-05)  & 0.142  & 19  & 5861  & 662  & {\color{red} 52}  & 07:58  & 01:52:34  & 8.3e-07  & 1.0e-06  & 1.1e-06 \\
&&&&&&&&&&&\\[-0.2cm]
 {\it NGrec}           &          (5e-03,5e-04)  & 0.004  & 26  & 8842  & 5  & {\color{red} 01:30}  & 07:17  & 03:34  & 6.4e-07  & 1.0e-06  & 4.7e-07 \\
 train                 &          (1e-03,1e-04)  & 0.067  & 24  & 8737  & 17  & {\color{red} 01:58}  & 09:31  & 06:32  & 7.6e-07  & 1.0e-06  & 1.7e-06 \\
 (300,4)               &          (5e-04,5e-05)  & 0.119  & 24  & 8625  & 66  & {\color{red} 02:05}  & 09:30  & 14:16  & 7.1e-07  & 1.0e-06  & 2.1e-06 \\
&&&&&&&&&&&\\[-0.2cm]
 {\it NGsci}           &          (5e-03,5e-04)  & 0.011  & 21  & 11166  & 10  & {\color{red} 41}  & 11:15  & 11:32  & 7.9e-07  & 1.0e-06  & 2.1e-08 \\
 train                 &          (1e-03,1e-04)  & 0.085  & 20  & 11137  & 41  & {\color{red} 01:03}  & 12:17  & 13:45  & 9.4e-07  & 1.0e-06  & 1.5e-06 \\
 (300,4)               &          (5e-04,5e-05)  & 0.146  & 20  & 11405  & 226  & {\color{red} 01:02}  & 12:41  & 28:27  & 9.1e-07  & 1.0e-06  & 1.6e-06 \\
&&&&&&&&&&&\\[-0.2cm]
 {\it NGtalk}          &          (5e-03,5e-04)  & 0.026  & 22  & 20000  & 12  & {\color{red} 01:29}  & 21:04  & 13:18  & 9.4e-07  & 2.0e-06  & 1.4e-06 \\
 train                 &          (1e-03,1e-04)  & 0.101  & 22  & 20000  & 74  & {\color{red} 41}  & 22:10  & 12:00  & 6.5e-07  & 1.7e-06  & 1.6e-06 \\
 (300,3)               &          (5e-04,5e-05)  & 0.193  & 22  & 20000  & 402  & {\color{red} 40}  & 23:15  & 45:12  & 6.2e-07  & 1.6e-06  & 1.6e-06 \\
\bottomrule
\end{tabular}}
\end{table}

Supplementary Table~\ref{table-ng300}  shows the comparison of rPPA, ADMM, and MGL  on the testing and training data sets of four subgroups with parameters $p=300$. The results for $p=200$ and $p=100$ are not shown by tables here for lack of space. Instead, we summarize all conducted instances (with different dimensionality $p=100$, $p=200$, and $p=300$, with various tuning parameters $(\lambda_1,\lambda_2)$) in Supplementary Figure~\ref{Fig-NG20-profile}. One can clearly see that rPPA outperforms ADMM and MGL by an obvious margin. It truly suggests that our proposed algorithm is efficient for solving the FGL problems.
\begin{figure}[!ht]
\centering
\centering
\includegraphics[width=0.5\textwidth]{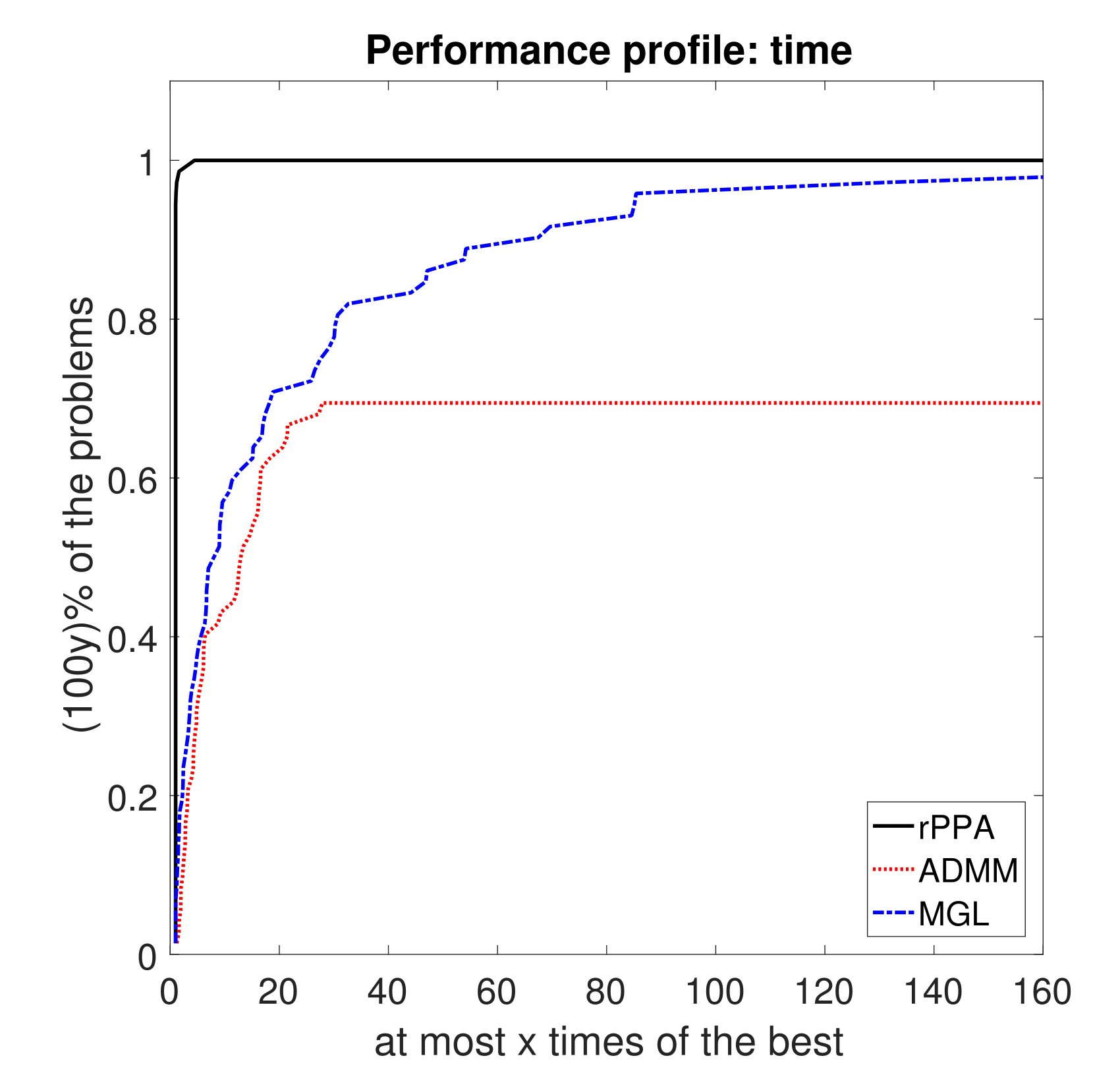}
\caption{\small{Performance profiles of rPPA, ADMM, and MGL on  newsgroups data sets with $p=100$, $p=200$, and $p=300$.}}\label{Fig-NG20-profile}
\end{figure}
\end{appendices}
\bibliographystyle{Chicago}

\begin{thebibliography}{}

\bibitem[\protect\citeauthoryear{Ahmed and Xing}{Ahmed and
  Xing}{2009}]{ahmed2009recovering}
Ahmed, A. and E.~P. Xing (2009).
\newblock Recovering time-varying networks of dependencies in social and
  biological studies.
\newblock {\em Proceedings of the National Academy of Sciences\/}~{\em
  106\/}(29), 11878--11883.

\bibitem[\protect\citeauthoryear{Banerjee, Ghaoui, and d'Aspremont}{Banerjee
  et~al.}{2008}]{banerjee2008model}
Banerjee, O., L.~E. Ghaoui, and A.~d'Aspremont (2008).
\newblock Model selection through sparse maximum likelihood estimation for
  multivariate {G}aussian or binary data.
\newblock {\em Journal of Machine Learning Research\/}~{\em 9}, 485--516.

\bibitem[\protect\citeauthoryear{Borwein and Lewis}{Borwein and
  Lewis}{2010}]{borwein2010convex}
Borwein, J. and A.~S. Lewis (2010).
\newblock {\em Convex analysis and nonlinear optimization: theory and
  examples}.
\newblock Springer Science \& Business Media.

\bibitem[\protect\citeauthoryear{Cardoso-Cachopo}{Cardoso-Cachopo}{2007}]{ana2007improving}
Cardoso-Cachopo, A. (2007).
\newblock Improving methods for single-label text categorization.
\newblock PhD Thesis, Instituto Superior Tecnico, Universidade Tecnica de
  Lisboa.

\bibitem[\protect\citeauthoryear{Cui, Sun, and Toh}{Cui
  et~al.}{2018}]{cui2018r}
Cui, Y., D.~F. Sun, and K.-C. Toh (2018).
\newblock On the {R}-superlinear convergence of the {KKT} residuals generated
  by the augmented {Lagrangian} method for convex composite conic programming.
\newblock {\em Mathematical Programming, DOI: 10.1007/s10107-018-1300-6\/}.

\bibitem[\protect\citeauthoryear{Danaher, Wang, and Witten}{Danaher
  et~al.}{2014}]{danaher2014joint}
Danaher, P., P.~Wang, and D.~M. Witten (2014).
\newblock The joint graphical lasso for inverse covariance estimation across
  multiple classes.
\newblock {\em Journal of the Royal Statistical Society: Series B (Statistical
  Methodology)\/}~{\em 76\/}(2), 373--397.

\bibitem[\protect\citeauthoryear{Facchinei and Pang}{Facchinei and
  Pang}{2007}]{facchinei2007finite}
Facchinei, F. and J.-S. Pang (2007).
\newblock {\em Finite-dimensional Variational Inequalities and Complementarity
  Problems}.
\newblock Springer Science \& Business Media.

\bibitem[\protect\citeauthoryear{Fan and Lv}{Fan and
  Lv}{2010}]{fan2010selective}
Fan, J. and J.~Lv (2010).
\newblock A selective overview of variable selection in high dimensional
  feature space.
\newblock {\em Statistica Sinica\/}~{\em 20\/}(1), 101.

\bibitem[\protect\citeauthoryear{Fan and Tang}{Fan and
  Tang}{2013}]{fan2013tuning}
Fan, Y. and C.~Y. Tang (2013).
\newblock Tuning parameter selection in high dimensional penalized likelihood.
\newblock {\em Journal of the Royal Statistical Society: Series B\/}~{\em
  75\/}(3), 531--552.

\bibitem[\protect\citeauthoryear{Friedman, Hastie, and Tibshirani}{Friedman
  et~al.}{2008}]{friedman2008sparse}
Friedman, J., T.~Hastie, and R.~Tibshirani (2008).
\newblock Sparse inverse covariance estimation with the graphical lasso.
\newblock {\em Biostatistics\/}~{\em 9\/}(3), 432--441.

\bibitem[\protect\citeauthoryear{Gibberd and Nelson}{Gibberd and
  Nelson}{2017}]{gibberd2017regularized}
Gibberd, A.~J. and J.~D. Nelson (2017).
\newblock Regularized estimation of piecewise constant gaussian graphical
  models: The group-fused graphical lasso.
\newblock {\em Journal of Computational and Graphical Statistics\/}~{\em
  26\/}(3), 623--634.

\bibitem[\protect\citeauthoryear{Guo, Levina, Michailidis, and Zhu}{Guo
  et~al.}{2011}]{guo2011joint}
Guo, J., E.~Levina, G.~Michailidis, and J.~Zhu (2011).
\newblock Joint estimation of multiple graphical models.
\newblock {\em Biometrika\/}~{\em 98\/}(1), 1--15.

\bibitem[\protect\citeauthoryear{Hallac, Park, Boyd, and Leskovec}{Hallac
  et~al.}{2017}]{hallac2017network}
Hallac, D., Y.~Park, S.~Boyd, and J.~Leskovec (2017).
\newblock Network inference via the time-varying graphical lasso.
\newblock In {\em Proceedings of the 23rd ACM SIGKDD International Conference
  on Knowledge Discovery and Data Mining}, pp.\  205--213. ACM.

\bibitem[\protect\citeauthoryear{Han, Sun, and Zhang}{Han
  et~al.}{2018}]{han2018linear}
Han, D., D.~F. Sun, and L.~Zhang (2018).
\newblock Linear rate convergence of the alternating direction method of
  multipliers for convex composite programming.
\newblock {\em Mathematics of Operations Research\/}~{\em 43\/}(2), 622--637.

\bibitem[\protect\citeauthoryear{Hsieh, Dhillon, Ravikumar, and Sustik}{Hsieh
  et~al.}{2011}]{hsieh2011sparse}
Hsieh, C.-J., I.~S. Dhillon, P.~K. Ravikumar, and M.~A. Sustik (2011).
\newblock Sparse inverse covariance matrix estimation using quadratic
  approximation.
\newblock In {\em Advances in Neural Information Processing Systems}, pp.\
  2330--2338. Curran Associates, Inc.

\bibitem[\protect\citeauthoryear{Lee, Sun, and Saunders}{Lee
  et~al.}{2014}]{lee2014proximal}
Lee, J.~D., Y.~Sun, and M.~A. Saunders (2014).
\newblock Proximal {N}ewton-type methods for minimizing composite functions.
\newblock {\em SIAM Journal on Optimization\/}~{\em 24\/}(3), 1420--1443.

\bibitem[\protect\citeauthoryear{Lemar{\'e}chal and
  Sagastiz{\'a}bal}{Lemar{\'e}chal and
  Sagastiz{\'a}bal}{1997}]{lemarechal1997practical}
Lemar{\'e}chal, C. and C.~Sagastiz{\'a}bal (1997).
\newblock Practical aspects of the {Moreau--Yosida} regularization: Theoretical
  preliminaries.
\newblock {\em SIAM Journal on Optimization\/}~{\em 7\/}(2), 367--385.

\bibitem[\protect\citeauthoryear{Li and Gui}{Li and Gui}{2006}]{li2006gradient}
Li, H. and J.~Gui (2006).
\newblock Gradient directed regularization for sparse {Gaussian} concentration
  graphs, with applications to inference of genetic networks.
\newblock {\em Biostatistics\/}~{\em 7\/}(2), 302--317.

\bibitem[\protect\citeauthoryear{Li, Sun, and Toh}{Li
  et~al.}{2018a}]{li2018highly}
Li, X., D.~F. Sun, and K.-C. Toh (2018a).
\newblock A highly efficient semismooth {Newton} augmented {Lagrangian} method
  for solving lasso problems.
\newblock {\em SIAM Journal on Optimization\/}~{\em 28\/}(1), 433--458.

\bibitem[\protect\citeauthoryear{Li, Sun, and Toh}{Li
  et~al.}{2018b}]{li2017efficiently}
Li, X., D.~F. Sun, and K.-C. Toh (2018b).
\newblock On efficiently solving the subproblems of a level-set method for
  fused lasso problems.
\newblock {\em SIAM Journal on Optimization\/}~{\em 28\/}(2), 1842--1862.

\bibitem[\protect\citeauthoryear{Monti, Hellyer, Sharp, Leech, Anagnostopoulos,
  and Montana}{Monti et~al.}{2014}]{monti2014estimating}
Monti, R.~P., P.~Hellyer, D.~Sharp, R.~Leech, C.~Anagnostopoulos, and
  G.~Montana (2014).
\newblock Estimating time-varying brain connectivity networks from functional
  {MRI} time series.
\newblock {\em NeuroImage\/}~{\em 103}, 427--443.

\bibitem[\protect\citeauthoryear{Moreau}{Moreau}{1965}]{moreau1965proximite}
Moreau, J.-J. (1965).
\newblock Proximit{\'e} et dualit{\'e} dans un espace hilbertien.
\newblock {\em Bull. Soc. Math. France\/}~{\em 93\/}(2), 273--299.

\bibitem[\protect\citeauthoryear{Rockafellar}{Rockafellar}{1976}]{rockafellar1976monotone}
Rockafellar, R.~T. (1976).
\newblock Monotone operators and the proximal point algorithm.
\newblock {\em SIAM Journal on Control and Optimization\/}~{\em 14\/}(5),
  877--898.

\bibitem[\protect\citeauthoryear{Rockafellar}{Rockafellar}{2015}]{rockafellar2015convex}
Rockafellar, R.~T. (2015).
\newblock {\em Convex Analysis}.
\newblock Princeton University Press.

\bibitem[\protect\citeauthoryear{Rockafellar and Wets}{Rockafellar and
  Wets}{2009}]{rockafellar2009variational}
Rockafellar, R.~T. and R.~J.-B. Wets (2009).
\newblock {\em Variational Analysis}, Volume 317.
\newblock Springer Science \& Business Media.

\bibitem[\protect\citeauthoryear{Rothman, Bickel, Levina, and Zhu}{Rothman
  et~al.}{2008}]{rothman2008sparse}
Rothman, A.~J., P.~J. Bickel, E.~Levina, and J.~Zhu (2008).
\newblock Sparse permutation invariant covariance estimation.
\newblock {\em Electronic Journal of Statistics\/}~{\em 2}, 494--515.

\bibitem[\protect\citeauthoryear{Tibshirani, Saunders, Rosset, Zhu, and
  Knight}{Tibshirani et~al.}{2005}]{tibshirani2005sparsity}
Tibshirani, R., M.~Saunders, S.~Rosset, J.~Zhu, and K.~Knight (2005).
\newblock Sparsity and smoothness via the fused lasso.
\newblock {\em Journal of the Royal Statistical Society: Series B\/}~{\em
  67\/}(1), 91--108.

\bibitem[\protect\citeauthoryear{Wang, Sun, and Toh}{Wang
  et~al.}{2010}]{wang2010solving}
Wang, C.~J., D.~F. Sun, and K.-C. Toh (2010).
\newblock Solving log-determinant optimization problems by a {Newton-CG} primal
  proximal point algorithm.
\newblock {\em SIAM Journal on Optimization\/}~{\em 20\/}(6), 2994--3013.

\bibitem[\protect\citeauthoryear{Yang and Peng}{Yang and
  Peng}{2018}]{yang2018estimating}
Yang, J. and J.~Peng (2018).
\newblock Estimating time-varying graphical models.
\newblock {\em arXiv preprint arXiv:1804.03811\/}.

\bibitem[\protect\citeauthoryear{Yang, Sun, and Toh}{Yang
  et~al.}{2013}]{yang2013proximal}
Yang, J.~F., D.~F. Sun, and K.-C. Toh (2013).
\newblock A proximal point algorithm for log-determinant optimization with
  group lasso regularization.
\newblock {\em SIAM Journal on Optimization\/}~{\em 23\/}(2), 857--893.

\bibitem[\protect\citeauthoryear{Yang, Lu, Shen, Wonka, and Ye}{Yang
  et~al.}{2015}]{yang2015fused}
Yang, S., Z.~Lu, X.~Shen, P.~Wonka, and J.~Ye (2015).
\newblock Fused multiple graphical lasso.
\newblock {\em SIAM Journal on Optimization\/}~{\em 25\/}(2), 916--943.

\bibitem[\protect\citeauthoryear{Yosida}{Yosida}{1964}]{yosida1980functional}
Yosida, K. (1964).
\newblock Functional analysis.
\newblock {\em Springer Berlin\/}.

\bibitem[\protect\citeauthoryear{Zhang, Zhang, Sun, and Toh}{Zhang
  et~al.}{2019}]{Zhang2018efficient}
Zhang, Y., N.~Zhang, D.~F. Sun, and K.-C. Toh (2019).
\newblock An efficient {H}essian based algorithm for solving large-scale sparse
  group {L}asso problems.
\newblock {\em Mathematical Programming, DOI: 10.1007/s10107-018-1329-6\/}.

\end{thebibliography}

\begin{thebibliography}{}

\bibitem[\protect\citeauthoryear{Chen, Sun, and Toh}{Chen
  et~al.}{2017}]{chen2017efficient}
Chen, L., D.~F. Sun, and K.-C. Toh (2017).
\newblock An efficient inexact symmetric {G}auss--{S}eidel based majorized
  {ADMM} for high-dimensional convex composite conic programming.
\newblock {\em Mathematical Programming\/}~{\em 161\/}(1-2), 237--270.

\bibitem[\protect\citeauthoryear{Fazel, Pong, Sun, and Tseng}{Fazel
  et~al.}{2013}]{fazel2013hankel}
Fazel, M., T.~K. Pong, D.~F. Sun, and P.~Tseng (2013).
\newblock Hankel matrix rank minimization with applications to system
  identification and realization.
\newblock {\em SIAM Journal on Matrix Analysis and Applications\/}~{\em
  34\/}(3), 946--977.

\bibitem[\protect\citeauthoryear{Friedman, Hastie, H{\"o}fling, Tibshirani,
  et~al.}{Friedman et~al.}{2007}]{friedman2007pathwise}
Friedman, J., T.~Hastie, H.~H{\"o}fling, R.~Tibshirani, et~al. (2007).
\newblock Pathwise coordinate optimization.
\newblock {\em The Annals of Applied Statistics\/}~{\em 1\/}(2), 302--332.

\bibitem[\protect\citeauthoryear{Gabay and Mercier}{Gabay and
  Mercier}{1976}]{gabay1976dual}
Gabay, D. and B.~Mercier (1976).
\newblock A dual algorithm for the solution of nonlinear variational problems
  via finite element approximation.
\newblock {\em Computers and Mathematics with Applications\/}~{\em 2\/}(1),
  17--40.

\bibitem[\protect\citeauthoryear{Glowinski and Marroco}{Glowinski and
  Marroco}{1975}]{glowinski1975approximation}
Glowinski, R. and A.~Marroco (1975).
\newblock Sur l'approximation, par {\'e}l{\'e}ments finis d'ordre un, et la
  r{\'e}solution, par p{\'e}nalisation-dualit{\'e} d'une classe de
  probl{\`e}mes de {D}irichlet non lin{\'e}aires.
\newblock {\em Revue fran{\c{c}}aise d'automatique, informatique, recherche
  op{\'e}rationnelle. Analyse num{\'e}rique\/}~{\em 9\/}(R2), 41--76.

\bibitem[\protect\citeauthoryear{Guo, Levina, Michailidis, and Zhu}{Guo
  et~al.}{2011}]{guo2011joint}
Guo, J., E.~Levina, G.~Michailidis, and J.~Zhu (2011).
\newblock Joint estimation of multiple graphical models.
\newblock {\em Biometrika\/}~{\em 98\/}(1), 1--15.

\bibitem[\protect\citeauthoryear{Lam, Marron, Sun, and Toh}{Lam
  et~al.}{2018}]{lam2018fast}
Lam, X.~Y., J.~Marron, D.~F. Sun, and K.-C. Toh (2018).
\newblock Fast algorithms for large-scale generalized distance weighted
  discrimination.
\newblock {\em Journal of Computational and Graphical Statistics\/}~{\em
  27\/}(2), 368--379.

\bibitem[\protect\citeauthoryear{Li, Sun, and Toh}{Li
  et~al.}{2018}]{li2018efficiently}
Li, X., D.~F. Sun, and K.-C. Toh (2018).
\newblock On efficiently solving the subproblems of a level-set method for
  fused lasso problems.
\newblock {\em SIAM Journal on Optimization\/}~{\em 28\/}(2), 1842--1862.

\bibitem[\protect\citeauthoryear{Qi}{Qi}{1993}]{qi1993convergence}
Qi, L. (1993).
\newblock Convergence analysis of some algorithms for solving nonsmooth
  equations.
\newblock {\em Mathematics of Operations Research\/}~{\em 18\/}(1), 227--244.

\bibitem[\protect\citeauthoryear{Tibshirani, Saunders, Rosset, Zhu, and
  Knight}{Tibshirani et~al.}{2005}]{tibshirani2005sparsity}
Tibshirani, R., M.~Saunders, S.~Rosset, J.~Zhu, and K.~Knight (2005).
\newblock Sparsity and smoothness via the fused lasso.
\newblock {\em Journal of the Royal Statistical Society: Series B (Statistical
  Methodology)\/}~{\em 67\/}(1), 91--108.

\end{thebibliography}

\end{document}